\documentclass[reqno, 12pt, fleqn]{amsart}
\usepackage[utf8]{inputenc}

\usepackage{graphicx}
\usepackage{amsmath,amssymb,amsthm, mathrsfs, mathtools,esvect}
\usepackage{xcolor}
\usepackage[backref]{hyperref}
\hypersetup{
    colorlinks,
    linkcolor={red!60!black},
    citecolor={green!60!black},
    urlcolor={blue!60!black}
}
\usepackage[T1]{fontenc}
\usepackage{lmodern}
\usepackage[babel]{microtype}
\usepackage[english]{babel}
\usepackage{comment}

\usepackage{geometry}
\numberwithin{equation}{section}
\numberwithin{figure}{section}

\usepackage{enumitem}

\theoremstyle{plain}
\newtheorem{thm}{Theorem}[section]
\newtheorem{theorem}[thm]{Theorem}

\newtheorem{obs}[thm]{Observation}
\newtheorem{clm}[thm]{Claim}
\newtheorem{cor}[thm]{Corollary}
\newtheorem{rem}[thm]{Remark}
\newtheorem{lemma}[thm]{Lemma}
\newtheorem{conjecture}[thm]{Conjecture}
\newtheorem{question}[thm]{Question}
\newtheorem{definition}[thm]{Definition}

\newcommand{\dist}{\text{dist}}
\newcommand{\polylog}{\text{polylog}}

\def\itm#1{\rm ({#1})}

\def\itmarab#1{\mbox{\itm{{\it #1\,}\arabic{*}\hspace{.05em}}}}

\DeclarePairedDelimiter{\set}{\{}{\}}

\title{Minimal Ramsey graphs with many vertices of small degree}

\author{Simona Boyadzhiyska
\and
Dennis Clemens
\and
Pranshu Gupta
}

\address{Institut f\"ur Mathematik, Freie Universit\"at Berlin, Berlin, Germany}
\email{s.boyadzhiyska@fu-berlin.de}
\thanks{The first author was supported by the Deutsche Forschungsgemeinschaft (DFG) Graduiertenkolleg ``Facets of Complexity'' (GRK 2434).}

\address{Institut f\"ur Mathematik, Technische Universit\"at Hamburg, Hamburg, Germany}
\email{dennis.clemens@tuhh.de, pranshu.gupta@tuhh.de}

\begin{document}
\maketitle 

\begin{abstract}
Given any graph $H$, a graph $G$ is said to be \emph{$q$-Ramsey} for $H$ if every coloring of the edges of $G$ with $q$ colors yields a monochromatic subgraph isomorphic to $H$. Further, such a graph $G$
is said to be \emph{minimal $q$-Ramsey} for $H$ if additionally
no proper subgraph $G'$ of $G$ is $q$-Ramsey for $H$. In 1976, Burr, Erd\H{o}s, and Lov\'asz initiated the study of the parameter $s_q(H)$, defined as the smallest minimum degree among all minimal $q$-Ramsey graphs for $H$.
In this paper, we consider the problem of determining how many vertices of degree $s_q(H)$ a minimal $q$-Ramsey graph for $H$ can contain. Specifically, we seek to identify graphs for which a minimal $q$-Ramsey graph can contain arbitrarily many such vertices. We call a graph satisfying this property \emph{$s_q$-abundant}. Among other results, we prove that every cycle is $s_q$-abundant for any integer $q\geq 2$. We also discuss the cases when $H$ is a clique or a clique with a pendant edge, extending previous results of Burr et al.\ and Fox et al. To prove our results and construct suitable minimal Ramsey graphs, we develop certain new gadget graphs, called \emph{pattern gadgets},
which generalize and extend earlier constructions that have proven useful in the study of minimal Ramsey graphs. These new gadgets might be of independent interest.
\end{abstract}

\section{Introduction}\label{sec:intro}

A classical result of F.~P.~Ramsey from 1930~\cite{Ram30} states that, for every graph $H$, there exists an integer $n$ such that the following property holds: No matter how the edges of $K_n$ are colored with two colors, there must always exist a \emph{monochromatic} copy of $H$, that is, a subgraph of $K_n$ isomorphic to $H$ in which all edges have the same color. 
In fact, the same is true if, instead of two, we use any arbitrary number of colors. 
Through the last decades, this result has become the starting point of a field of intense studies, giving rise to a branch of combinatorics known as  \emph{Ramsey theory}. For an excellent survey on the more recent developments in the field, see~\cite{conlon2015recent}.

One line of research is concerned with studying properties of (minimal) Ramsey graphs, which is also the focus of this paper.  
Formally, given any graph $H$, a graph $G$ is said to be \emph{$q$-Ramsey} for $H$, denoted by $G \rightarrow_q H$, if in every coloring of the edges of $G$ with $q$ colors, there exists a monochromatic copy of $H$. Ramsey's theorem discussed above then states that, for every graph $H$, we have $K_n \rightarrow_q H$ provided that $n$ is large enough, thus establishing the existence of a $q$-Ramsey graph for $H$ for any choice of $H$ and $q\geq 2$. We denote the set of all such graphs for $H$ by $\mathcal{R}_q(H)$.

In this language, the well-known \emph{$q$-color Ramsey number} of a graph $H$, denoted by $r_q(H)$, can be defined as the minimum possible number of vertices in a graph that is $q$-Ramsey for $H$. Over the years, researchers have worked hard to understand the behavior of Ramsey numbers for various classes of graphs, which in some cases has turned out to be notoriously difficult. Perhaps the most natural example here is the clique $K_t$. While the determination of $r_2(K_3)$ is a simple exercise often given in a first course in combinatorics, already for $t = 5$, the precise value of $r_2(K_t)$ is not known. For general $t$, Erd\H{o}s and Szekeres~\cite{erdos1935combinatorial} and Erd\H{o}s~\cite{erdos1947some} showed that $2^{t/2} \leq r_2(K_t)\leq 2^{2t}$, establishing that the $2$-color Ramsey number of $K_t$ is exponential in $t$ but leaving a large gap between the two bounds in the base of the exponent. Now, more than 70 years later, those remain essentially the best known bounds, with improvements only in the lower order terms.
The current best known lower bound is due to Spencer~\cite{spencer1975ramsey}; a new upper bound was shown very recently by Sah~\cite{sah2020diagonal}, improving on the previous best known bound due to  Conlon~\cite{conlon2009new}.

More generally, it is of interest to understand what makes a graph $q$-Ramsey for some chosen graph $H$, that is, to understand the structural properties of graphs that are $q$-Ramsey for $H$ and, whenever possible, to characterize all such graphs. After considering the number of vertices, it is natural to ask about the behavior of other graph parameters. For example, much work has been done in studying the minimum possible number of edges in a graph that is $q$-Ramsey for $H$, known as the \emph{$q$-color size-Ramsey number} of $H$.

Here, we are interested in questions concerning minimum degrees of graphs that are $q$-Ramsey for a graph $H$. Note that asking about the smallest possible minimum degree of a graph that is $q$-Ramsey for $H$ is not very interesting, as we can immediately see that the answer is zero. This is because any graph containing a $q$-Ramsey graph for $H$ as a subgraph is itself $q$-Ramsey for $H$, and we can of course add an isolated vertex to obtain a graph with minimum degree zero. To avoid such trivialities, we restrict our attention to those graphs that are, in some sense, critically $q$-Ramsey for $H$. This leads to the following natural definition: We say $G$ is \emph{minimal $q$-Ramsey} for $H$ if $G \rightarrow_q H$ and, for any proper subgraph $G' \subsetneq G$, we have $G' \not\rightarrow_q H$, that is, $G$ loses its Ramsey property whenever we delete any vertex or edge of $G$. We denote the set of all minimal $q$-Ramsey graphs for $H$ by $\mathcal{M}_q(H)$.

The 1970s saw the beginning of two prominent directions of research concerning $\mathcal{M}_q(H)$. 
One of the questions, first posed in~\cite{nevsetvril1976partitions} by Ne{\v{s}}et{\v{r}}il and R{\"o}dl, 
was whether for a given graph $H$ the set $\mathcal{M}_q(H)$ is finite or infinite. 
We call a graph $H$ \emph{$q$-Ramsey finite} (respectively \emph{infinite}) if the set $\mathcal{M}_q(H)$ is finite (respectively infinite). In 1978, Ne{\v{s}}et{\v{r}}il and R{\"o}dl~\cite{nevsetvril1978structure} showed that $H$ is 2-Ramsey infinite in the following three cases: if $H$ is not bipartite, if $H$ is 2.5-connected (that is, $H$ is 2-connected and the removal of any pair of adjacent vertices does not disconnect it), and if $H$ is a forest containing a path of length three. In~\cite{burr1978class} the authors showed that any matching is $2$-Ramsey finite. Subsequently, in~\cite{BURR1981227stars}, Burr et al.\ showed that if $H$ is a disjoint union of non-trivial stars
(i.e., all stars have at least two edges), then $H$ is 2-Ramsey finite if and only if $H$ is an odd star; they also showed that the disjoint union of an odd star with any number of isolated edges is 2-Ramsey-finite.  
Finally, the results of R{\"o}dl and Ruci{\'n}ski~\cite{rodl1995threshold} imply that every graph containing a cycle is 2-Ramsey infinite. As a result, we know that a graph $H$ is 2-Ramsey-finite if and only if it is the disjoint union of an odd star and any number of isolated edges. 

Around the same time, Burr, Erd\H{o}s and Lov{\'a}sz~\cite{burr1976graphs} initiated the general study of graph parameters for graphs in
$\mathcal{M}_q(H)$. In their seminal paper,
they considered the chromatic number, the (vertex) connectivity, and the minimum degree of minimal $2$-Ramsey graphs for the clique $K_t$, $t\geq 3$. In particular, they were interested in how small these parameters can be.
 
Surprisingly, while the $2$-Ramsey number of $K_t$ is still not known, Burr et al.~could determine the mentioned values precisely.

Following~\cite{fox2016minimum}, we set $s_q(H) = \min \{ \delta(G) : G \in \mathcal{M}_q (H) \}$, where as usual $\delta(G)$ denotes the minimum degree of $G$. 
One of the results appearing in~\cite{burr1976graphs} establishes that $s_2(K_t)=(t-1)^2$, which is perhaps surprising, given that each graph in $\mathcal{M}_q(K_t)$ has at least exponentially many vertices.

For more colors, Fox et al.~\cite{fox2016minimum} established that $s_q(K_t) \leq 8(t-1)^6q^3$, showing that $s_q(K_t)$ is polynomial in both $t$ and $q$. 
Moreover, they also investigated the growth of $s_q(K_t)$ as a function of $q$ (with $t$ being treated as a constant) and proved that $s_q(K_t) = q^2\polylog (q)$. However, a logarithmic gap remained between the lower and the upper bound. For the case of the triangle, Guo and Warnke~\cite{gw2020triangle} closed this
gap, showing that
$s_q(K_3)=\Theta(q^2\log q)$.
On the other hand, Hàn et al.~\cite{hanrodlszabo2018} studied the dependence of $s_q(K_t)$ on the size of the clique with the number of colors kept constant; they showed that $s_q(K_t) = t^2\polylog (t)$.

The parameter $s_q(H)$ has also been investigated for other choices of the target graph $H$ when $q=2$. For instance, Szab\'o et al.~\cite{szabo2010minimum} determined $s_2(H)$ for many interesting classes of bipartite graphs, including trees, even cycles, and biregular bipartite graphs. Later Grinshpun~\cite{grinshpun2015some} determined $s_2(H)$ for any 3-connected bipartite graph $H$. A rather surprising result in this direction appeared in a paper of Fox et al.~\cite{fox2014ramsey}, who studied $s_2(K_t\cdot K_2)$, where $K_t\cdot K_2$ is the graph obtained from a clique of size $t$ by adding a new vertex and connecting it to exactly one vertex of the clique (we will call such a graph a \emph{clique with a pendant edge}). The authors proved that 
$s_2(K_t\cdot K_2) = t-1$, showing that even a single edge can significantly change the value of the parameter $s_2$. This result also implies that there exists a $2$-Ramsey graph for $K_t$ that is not $2$-Ramsey for $K_t\cdot K_2$.

Once we know that a minimal $q$-Ramsey graph for a given $H$ can contain a vertex of small degree, a natural next question is, \emph{how many} vertices of this small degree can a minimal $q$-Ramsey graph for $H$ contain? More specifically, can a minimal $q$-Ramsey graph have arbitrarily many vertices of the smallest possible minimum degree? This question motivates the following definition.

\begin{definition}\rm
For a given integer $q \geq 2$, a graph $H$ is said to be \emph{$s_q$-abundant} if, for every $k \geq 1$, there exists a minimal $q$-Ramsey graph for $H$ with at least $k$ vertices of degree $s_q(H)$.
\end{definition}

As it turns out, it is not immediate whether $s_q$-abundant graphs exist at all. In this paper, we will give several examples showing that, for all $q\geq 2$, there are infinitely many $s_q$-abundant graphs.

It is not hard to see that, if a graph is $q$-Ramsey finite, then it cannot be $s_q$-abundant. 
This immediately implies that odd stars are not $s_2$-abundant. On the other hand, we know that even stars \emph{are} $2$-Ramsey infinite, but as we will see below they are also not $s_2$-abundant.
This statement follows from the following result of Burr et al.~\cite{burr1976graphs}.

\begin{theorem}[{\cite{burr1976graphs}}]
Let $m\geq 1$ be an integer. Then a connected graph $G$ is $2$-Ramsey for $K_{1,m}$ if and only if either $\Delta(G) \geq 2m-1$ or $m$ is even and $G$ is a $(2m-2)$-regular graph on an odd number of vertices.
\end{theorem}

The theorem immediately implies that $\mathcal{M}_2(K_{1,m}) = \set{K_{1,2m-1}}$ if $m$ is odd and $\mathcal{M}_2(K_{1,m}) = \set{K_{1,2m-1}}\cup \set{G : G \text{ is connected and }(2m-2)\text{-regular and } |V(G)|\text{ is odd}}$ if $m$ is even. In particular, this implies that no star is $s_2$-abundant.

More generally, it turns out that stars are not $s_q$-abundant for any $q\geq 2$:
A simple argument implies that, for any $m\geq 1$ and $q\geq 2$, a minimal $q$-Ramsey graph for $K_{1,m}$ has either zero or $q(m-1)+1$ vertices of degree one. Indeed, if $G$ is a minimal $q$-Ramsey graph for $K_{1,m}$ that is not isomorphic to $K_{1, q(m-1)+1}$, then the maximum degree of $G$ is at most $q(m-1)$. Thus, if $G$ contains a vertex $v$ of degree one, then the only neighbor $u$ of $v$ has at most $q(m-1)-1$ other neighbors. By the minimality of $G$, the graph $G-v$ has a $q$-coloring $c$ without a monochromatic copy of $K_{1,m}$. Since $u$ has at most $q(m-1)-1$ neighbors in $G-v$, there is a color that appears at most $m-2$ times on the edges incident to $u$. Then this color can be used on the edge $uv$ to extend $c$ to a $q$-coloring of $G$ without a monochromatic copy of $K_{1,m}$, leading to a contradiction. Hence, $G$ cannot contain a vertex of degree one. 

\medskip

One of the goals of this paper is to initiate the systematic study of $s_q$-abundance: We introduce a tool that can be used to show that a given graph is $s_q$-abundant, and we illustrate the 
utility of our tool by presenting a few applications. First, we show that all cycles of length at least four are $s_q$-abundant. As a byproduct, we determine $s_q(C_t)$ for all $q\geq 2$ and  $t\geq 4$.

\begin{thm}\label{cor:arbitrarily_many_cycles}
For any given integers $q\geq 2$, $t \geq 4$, and $k\geq 1$, there exists a minimal $q$-Ramsey graph for $C_t$ that has at least $k$ vertices of degree 
$q+1$. In particular, $s_q(C_t) = q+1$ and $C_t$ is $s_q$-abundant.
\end{thm}

It turns out that the cycle $C_3$ behaves differently compared to longer cycles with respect to the value of $s_q$. Its behavior is consistent with that of a clique, and we know from an earlier discussion that $s_2(K_3) = 4$ and $s_q(K_3) = \Theta (q^2 \log q)$ as a function of $q$.  While the value of $s_q$ for $K_3$, or for any larger clique, is not known precisely when $q>2$, our theory still allows us to show that any clique $K_t$ for $t\geq 3$ is $s_q$-abundant for any value of $q$. In fact, Theorem~\ref{cor:arbitrarily_many_cliques} below is a consequence of a more general result that will be presented in Section~\ref{sec:applications}.

\begin{thm}\label{cor:arbitrarily_many_cliques}
For any given integers $q\geq 2$ and $t\geq 3$, the clique $K_t$ is $s_q$-abundant.
\end{thm}

\iffalse
\begin{cor}\label{cor:arbitrarily_many_cliques_sufflarge}\hfill
\begin{enumerate}
    \item[(a)] For any integer $t\geq 3$, the clique $K_t$ is $s_2$-abundant.
    \item[(b)] For any integer $t\geq 3$, there exists $q_0 = q_0(t)$ such that $K_t$ is $s_q$-abundant for all $q\geq q_0$.
\end{enumerate}
\end{cor}

Indeed, following the result of Burr et al.~\cite{burr1976graphs} we know that
$s_2(K_t)=(t-1)^2$. On the other hand,
it holds that $r_2(K_t) \geq (t-1)^2 + 2$,
as we can define a $K_t$-free 2-coloring of $K_{(t-1)^2+1}$ as follows: Fix a partition
of the vertex set $V:=V(K_{(t-1)^2+1})=\{v\}\cup V_1\cup \ldots \cup V_{t-1}$ with $|V_i|=t-1$
and fix an arbitrary vertex $v_i\in V_i$
for every $i\in [t-1]$.
Give color red to $v_1v_2$, to every edge
$vw$ with $w\in V\setminus \{v,v_1,\ldots,v_{t-1}\}$ and
to every edge in $V_i$ for every $i\in [t-1]$.
Give color blue to every other edge.

Moreover, for any $t\geq 3$,
we obtain 
$s_q(K_t)\leq r_q(K_t)-2$ for large enough $q$,
since $s_q(K_t) \leq C(t-1)^5 q^{2.5}$ by~\cite{BBL2020}, while
by~\cite{abbott1972}
there exists some constant $c_t$
depending only on $t$ such that
$r_q(K_t) \geq c_t (2t-3)^q$. 
\fi

As a third application of our theory, we show that a clique with a pendant edge is $s_2$-abundant. We note that, since $s_2(K_t) = (t-1)^2$ and $s_2(K_t\cdot K_2) = t-1$ for all $t\geq 3$, Theorem~\ref{cor:arbitrarily_many_Kt.k2} also yields that there are infinitely many graphs that are minimal $2$-Ramsey for $K_t\cdot K_2$ but not minimal $2$-Ramsey for $K_t$. One of the main building blocks used in the construction of our main tool is not known to exist for $K_t\cdot K_2$ when $q>2$, which is why we focus on the case $q=2$. 

\begin{thm}\label{cor:arbitrarily_many_Kt.k2}
For a given integer $t \geq 3$, the graph $K_t \cdot K_2$ is $s_2$-abundant.
\end{thm}

In order to prove the statements above, we will first construct new gadget graphs, called \emph{pattern gadgets}. These generalize other well-known gadgets such as \emph{signal senders}, originally developed by Burr, Erd\H{o}s, and Lov\'asz~\cite{burr1976graphs} to study $s_2(K_t)$, and 
\emph{$2$-colorings gadgets} and 
\emph{one-in-m gadgets}, developed by Siggers in~\cite{siggers2014nonbipartite}.
Our pattern gadgets help us construct minimal Ramsey graphs with many vertices of small degree. Given the utility of signal senders in studying 
properties of the set $\mathcal{M}_q(H)$ for various graphs $H$, our pattern gadgets might also be of independent interest and find further applications.

In a nutshell, the main idea of our pattern gadgets will be the following: Given some graph $G$ and some family $\mathscr{G}$ of colorings of $E(G)$
with $q$ colors that do not contain monochromatic copies of $H$,
we will find some larger graph $P$ containing $G$ such that the colorings in $\mathscr{G}$
are exactly those colorings of $G$ that
can be extended to $P$ without creating a monochromatic copy of $H$.
Then, in order to prove each of the above theorems,
we will choose $G$ and $\mathscr{G}$ in such a way that we can attach $k$ small-degree vertices to $G\subseteq P$ so that no coloring in $\mathscr{G}$ can be extended to the new edges without creating a monochromatic copy of $H$, but if we remove any of these new vertices, we can find a coloring in $\mathscr{G}$
that can be extended in the desired way.

The precise definition of a pattern gadget will be given in Section~\ref{sec:pattern_gadgets_construction}. We will show their existence for many target graphs $H$, including all $3$-connected graphs.

\medskip

{\bf Organization of the paper.} 
In Section~\ref{sec:pattern_gadgets_construction},
we introduce all necessary auxiliary gadgets and prove the existence
of pattern gadgets.
Afterwards, we continue with the proofs of the
Theorems~\ref{cor:arbitrarily_many_cycles},~\ref{cor:arbitrarily_many_cliques}, \ref{cor:arbitrarily_many_Kt.k2} in Section~\ref{sec:applications}, where we also prove a general statement regarding $3$-connected graphs.  
We end with some concluding remarks and open problems in Section~\ref{sec:conclusion}. 
\medskip

{\bf Notation.} 
Given an integer $n\geq 1$, we write $[n]$ for the set of the first $n$ positive integers.
For a graph $G$, we denote its vertex set by $V(G)$ and its edge set by $E(G)$.
For any edge $\{v,w\}\in E(G)$, we write $vw$ for short.
We let 
$N_G(v)=\{w\in V(G):~ vw\in E(G)\}$
denote the neighborhood of $v$ in $G$, $d_G(v)=|N_G(v)|$ denote the degree of $v$ in $G$, $\delta(G)=\min\{d_G(v):~ v\in V(G)\}$, and $\Delta(G)=\max\{d_G(v):~ v\in V(G)\}$ denote the minimum degree and maximum degree of $G$ respectively. 

For a graph $G$ and vertex subsets $A$ and $B$ of $G$, we denote by $E_G(A,B)$ the edges in $G$ with one endpoint in $A$ and another in $B$. Also, $E_G(A)$ denotes the edges in $G$ with both endpoints in $A$. We sometimes identify a graph $G$ with its edge set.

Let $F$ and $G$ be two graphs. We say that $F$ and $G$ are isomorphic, denoted by
$F\cong G$,
if there exists a bijection $f:V(F)\rightarrow V(G)$
such that $vw \in E(F)$ if and only if $f(v)f(w)\in E(G)$. In this case, we also say that $F$ forms a copy of $G$.

We say that $F$ is a subgraph of $G$, denoted by $F\subseteq G$, if there is an injective map  $f: V(F)\rightarrow V(G)$ such that $f(x)f(y) \in E(G)$ for all $xy \in E(F)$; further, $F$ is a proper subgraph of $G$ if $F\subseteq G$ and $F\neq G$. 
Given any subset $A\subseteq V$, the subgraph induced by $A$, denoted $G[A]$, is the graph with vertex set $A$ and edge set $E_G(A)$.
Moreover, we set $G-v=G[V(G)\setminus {v}]$ and
$G-e=(V(G),E(G)\setminus \{e\})$ for any $v\in V(G)$ and $e\in E(G)$. 
If $F \cong G[A] $ for some $A \subseteq V(G)$, then we say that $F$ is an induced subgraph of $G$ and write $F\subseteq_{ind} G$.

Given a graph $G$ and any subsets $A$ and $B$
of the vertex set or the edge set of $G$, 
we define the {\em distance} between $A$ and $B$, denoted $\dist_G(A,B)$, to be the number of edges in a shortest path with one endpoint in (the vertex set of) $A$ and one endpoint in (the vertex set of) $B$.
The girth of $G$, denoted $\text{girth}(G)$,
is the length of a shortest cycle in $G$ (if $G$ is acyclic, then $\text{girth}(G)$ is defined to be infinity).
A graph $G$ is said to be $k$-connected if it has more than $k$ vertices and, for any set $S$ of at most $k-1$ vertices, the graph $G[V(G)\setminus S]$ is connected.
 
In the rest of the paper, a coloring of some graph $G$ always refers to a coloring of its edge set. If $G$ contains no monochromatic subgraph isomorphic to $H$ under a given coloring, the coloring is said to be $H$-free.
%Whenever we consider a coloring of some graph $G$ in this paper, we mean a coloring of its edge set. 
If a coloring uses at most $q$ colors, we call it a $q$-coloring. Unless otherwise specified, we will assume in this case that our color palette is the set $[q]$. If we are only concerned with the case $q=2$, for the sake of convenience we will sometimes call our colors red and blue instead of color 1 and color 2.
If $c$ is a $q$-coloring of $G$ and some subgraph $F$ is monochromatic in some color $i$, we will sometimes write $c(F) = i$. Similarly, when defining colorings, we will write for example $c(F) = i$ to indicate that we give color $i$ to every edge of the subgraph $F$.

\section{Construction of Pattern gadgets}\label{sec:pattern_gadgets_construction}

Most of our constructions of minimal Ramsey graphs will rely
on the existence of certain gadget graphs; these graphs will have the property that, in every coloring not containing a monochromatic copy of our target graph $H$, some fixed \emph{color patterns} need to appear on certain sets of edges.
Such an approach has already been used in the paper of
Burr et al.~\cite{burr1976graphs} when proving that $s_2(K_t)=(t-1)^2$. In their paper, the authors introduced gadget graphs that are now known as {BEL gadgets} and are defined as follows: Let $H$ and $G$ be fixed graphs such that $G\not\rightarrow_q H$, and let $\varphi$ be an $H$-free $q$-coloring of $G$; a \emph{BEL gadget} for $H$ with respect to the pair $(G, \varphi)$ is a graph $B$ containing $G$ as an induced subgraph such that $B$ is not $q$-Ramsey for $H$ but in every $H$-free $q$-coloring of the edges of $B$, the subgraph $G$ has the coloring given by $\varphi$ (up to a permutation of colors). Burr et al.\ showed the existence of BEL gadgets for all cliques on at least three vertices when $q=2$ (for any appropriate choice of $G$ and $\varphi$). Later results imply that BEL gadgets exist for more general graphs and for more colors; we will give an overview of those results in Section~\ref{sec:indicators}.

With such a gadget at hand, to construct a minimal $q$-Ramsey graph for $H$ that contains a vertex of degree at most $d$, it suffices to find a graph $G$ that contains a vertex $v$ of degree $d$ and a $q$-coloring $\varphi$ of $ G-v$ that contains no monochromatic copy of $H$ but cannot be extended to an $H$-free coloring of $G$. Indeed, we can construct $\widetilde{G}$ by taking a copy $G'$ of $G-v$ and a BEL gadget for $H$ with respect to $(G', \varphi)$ and adding the vertex $v$ along with $d$ edges so that $V(G')\cup\set{v}$ induces a copy of $G$.
Now it is not difficult to check that $\widetilde{G}\rightarrow_q H$, and if $H$ satisfies certain conditions, then we can also ensure that $\widetilde{G}-v\not\rightarrow_q H$. This means that any minimal $q$-Ramsey subgraph of $\widetilde{G}$ needs to contain $v$, that is, $v$ is \emph{important} for 
$\widetilde{G}$ to be a $q$-Ramsey graph, and $s_q(H)\leq d_{\widetilde{G}}(v)$.

For our main theorems, we will aim to find graphs $\widetilde{G}$ with many vertices of small degree, each of which is important for 
$\widetilde{G}$ to be a Ramsey graph for $H$.
In order to do so, we will construct a gadget that allows for more flexibility than a BEL gadget.
The new gadget again comes with a subgraph $G$ on which fixed color patterns 
are forced in any $H$-free $q$-coloring. However, while for a BEL gadget we fix only a single pattern (up to a permutation of the color classes),
our gadget graph allows us to fix a family of color patterns for $G$ such that each of these patterns, and no other, can be extended to an $H$-free coloring of the whole graph.

To make this more precise, let us first define color patterns and an isomorphism relation between them.

\begin{definition}\label{def:pattern}\rm 
Let $q \geq 2$ be a given integer and $H$ and $G$ be graphs. A \emph{$q$-color pattern} for $G$ is a partition $g=\{G_1,G_2,\ldots,G_q\}$ of the edges of $G$.
If $H\not\subseteq G_i$ for every $i\in [q]$, we say that $g$ is \emph{$H$-free}.
Given any subset $A\subseteq V(G)$, we call the partition $g[A]=\{G_1[A],G_2[A],\ldots,G_q[A]\}$
the \emph{induced $q$-color pattern} on $A$. 

Let $G'$ be a copy of $G$, and let $g'=\{G'_1,\ldots,G'_q\}$
be a $q$-color pattern for $G'$. Then we say that $g$ and $g'$
are \emph{isomorphic}, denoted $g\cong g'$,
if there exists a permutation $\pi$ of $[q]$
such that $G_i\cong G'_{\pi(i)}$ for every $i\in [q]$.
\end{definition}

Using the above terminology, we can now give a precise definition of the gadget graphs that we are interested in.

\begin{definition}\label{def:pattern_gadgets}\rm
Let $q\geq 2$ be a given integer and $H$ and $G$ be graphs such that $G\not\rightarrow_{q} H$. 
Also let $\mathscr{G}$ be a family of $H$-free $q$-color patterns for $G$. 
Then we call a graph $P = P(H, G, \mathscr{G}, q)$ a \emph{pattern gadget} if the following properties hold:
\begin{enumerate}[label=\itmarab{P}]
    \item\label{def:patterngadget_subgraphG} $G\subseteq_{ind}P$.
    \item\label{def:patterngadget_somepattern} If $c:E(P)\rightarrow [q]$ is an $H$-free coloring of $P$, then 
    $\{c_{|G}^{-1}(1), \dots, c_{|G}^{-1}(q)\} \in \mathscr{G}$.
    \item\label{def:patterngadget_fixedpattern} For every pattern in $\{G_1,\dots, G_q\}\in \mathscr{G}$, there exists an $H$-free coloring $c:E(P)\rightarrow [q]$ such that $\{c_{|G}^{-1}(1), \dots, c_{|G}^{-1}(q)\} = \{G_1,\dots, G_q\} $.
\end{enumerate}
\end{definition}

The rest of this section is mainly devoted to the proof that pattern gadgets exist for certain choices of the graph $H$.
In the proof, we will combine various intermediate gadgets and for that to work we will often require them to satisfy an additional property that we refer to as robustness (following Grinshpun~\cite{grinshpun2015some}).
We will also require that our final gadgets satisfy this property, which will be useful in applications.

\begin{definition}\label{def:robustness}\rm
Let $G$ be a graph and $G_0$ be an induced subgraph of $G$. We say that the pair $(G,G_0)$ is $H$-robust if, in any graph obtained from $G$ by adding any set $S$ of new vertices and any collection of edges within $S \cup V(G_0)$, every copy of $H$ is entirely contained either in $G$ or in the subgraph induced by $S \cup V(G_0)$. 
\end{definition}

The main theorem of this section states that, if $H$ is 3-connected or isomorphic to a cycle or $K_t\cdot K_2$,
then pattern gadgets that satisfy certain robustness properties exist for $H$.

\begin{theorem}\label{thm:pattern_gadgets_existence}
Let $q\geq 2$ be a given integer, and let $H$ and $G$ be graphs with $G\not\rightarrow_{q} H$. Further, let
$\mathscr{G}$ be a family of $H$-free 
$q$-color patterns for $G$.
\begin{enumerate}[label=(\alph*)]
    \item\label{thm:pattern_gadgets_existence:3connected} If $H$ is 3-connected or a triangle, then a pattern gadget $P = P(H, G, \mathscr{G}, q)$ exists.
    \item\label{thm:pattern_gadgets_existence:cycles} If $H$ is a cycle of length at least four, then a pattern gadget $P = P(H, G, \mathscr{G}, q)$ exists.
    \item\label{thm:pattern_gadgets_existence:pendant} If $H\cong K_t\cdot K_2$ and $q=2$ and $G$ does not contain a copy of $H$, then a pattern gadget $P = P(H, G, \mathscr{G}, q)$ exists. Further, we can ensure that in the $2$-colorings in \ref{def:patterngadget_fixedpattern} every monochromatic copy of $K_t$ using a vertex from $G$ is fully contained in $G$.
\end{enumerate}
Further, in parts \ref{thm:pattern_gadgets_existence:3connected} and \ref{thm:pattern_gadgets_existence:cycles}, the pattern gadget can be taken so that $(P,G)$ is $H$-robust, and in part \ref{thm:pattern_gadgets_existence:pendant}, it can be taken so that $(P,G)$ is $K_t$-robust. 
\end{theorem}

Before we give the proof of Theorem~\ref{thm:pattern_gadgets_existence}
in Section~\ref{sec:pattern_gadgets_existence},
we need to introduce two different simpler gadgets, known as {\em signal senders} and {\em indicators}, in
Section~\ref{sec:indicators},
and to construct a generalization 
of the latter, which we will call
{\em generalized negative indicators}, in Section~\ref{sec:gen_negative_indicators}.

\subsection{Signal senders and indicators}\label{sec:indicators}

Signal senders were introduced by Burr et.~al~\cite{burr1976graphs}
for the construction of BEL gadgets when $H\cong K_t$ for $t\geq3$ and $q=2$.

\begin{definition}\label{def:signal_senders}\rm
Let $q\geq 2$ and  $d\geq 1$ be given integers, and let $H$ be a graph. A {\em positive signal sender} $S = S^+(H,e,f,q,d)$ for $H$ is a graph that contains two distinguished edges $e,f\in E(S)$, called the {\em signal edges} of $S$, such that the following properties hold:
\begin{enumerate}[label=\itmarab{S}]
    \item\label{def:signal_senders:Hfree} $S\not\rightarrow_q H$.
    \item\label{def:signal_senders:signal} In any $H$-free $q$-coloring of $S$, the edges $e$ and $f$ have the same color.
    \item\label{def:signal_senders:dist} $\dist_S(e,f) \geq d$.
\end{enumerate}
 A {\em negative signal sender} $S = S^-(H,e,f,q,d)$ for $H$ is defined similarly, except that the words ``the same color'' in \ref{def:signal_senders:signal} are replaced by ``different colors.'' 
 
 An \emph{interior} vertex of a signal sender is a vertex that is not incident to either of the signal edges. The \emph{interior} of a signal sender is the set of all interior vertices. 
\end{definition}

Signal senders are known to exist for some important classes of graphs, as given by Theorem~\ref{thm:signal_senders_existence} below. Part \ref{thm:signal_senders_existence:3connected} is due to R\"odl and Siggers~\cite{rodl2008ramseyminimal}, generalizing results of Burr et al.~\cite{burr1976graphs} and Burr et al.~\cite{burr1985useofsenders}, part \ref{thm:signal_senders_existence:cycles} is due to Siggers~\cite{siggers2008cycles}, and part \ref{thm:signal_senders_existence:pendant} follows from a result in the PhD thesis of Grinshpun~\cite[Lemma 2.6.3]{grinshpun2015some} combined with the result of Fox et al.~\cite{fox2014ramsey} concerning $s_2(K_t\cdot K_2)$. 

\begin{thm}\label{thm:signal_senders_existence} \hfill
\begin{enumerate}[label=(\alph*)] 
    \item \label{thm:signal_senders_existence:3connected} For all integers $q\geq 2$ and $d\geq 1$ and every graph $H$ that is 3-connected or isomorphic to $K_3$, there exist positive and negative signal senders in which the distance between the signal edges is at least $d$.
    \item \label{thm:signal_senders_existence:cycles} For all integers $q\geq 2$, $d\geq 1$, and $t\geq 4$, there exist positive and negative signal senders for $C_t$ with girth $t$ and distance at least $d$ between the signal edges.
    \item \label{thm:signal_senders_existence:pendant} 
    For $q=2$ and for all integers $t\geq 3$ and $d\geq 1$, there exist positive and negative signal senders for $K_t\cdot K_2$ in which the distance between the signal edges is at least $d$. Further, a signal sender $S$, positive or negative, with signal edges $e$ and $f$ can be chosen so that $S$ has a $K_t\cdot K_2$-free $2$-coloring in which all edges incident to $e$ (resp.\ $f$) have a different color from $e$ (resp.\ $f$) and none of the vertices of $e$ and $f$ is contained in a monochromatic copy of $K_t$.
\end{enumerate}
\end{thm}

Before we continue, we make a few remarks about Theorem~\ref{thm:signal_senders_existence}. 
First, in~\cite{grinshpun2015some}, Grinshpun does not explicitly prove that signal senders exist for $K_3\cdot K_2$; however, his proof easily extends to this case. Further, part~\ref{thm:signal_senders_existence:pendant} is actually a slight strengthening of Grinshpun's result: His result is stated only in terms of negative signal senders and provides a special coloring in which neither signal edge is incident to a monochromatic copy of $K_t$ but only one of the signal edges, say $f$, is required to have a color different from all edges incident to it. We can derive the version stated above easily. Let $S'$ be the signal sender constructed by Grinshpun. To construct a positive signal sender $S^+$ as in Theorem~\ref{thm:signal_senders_existence:pendant}, take two copies of $S'$ and identify the two copies of $e$; similarly, to construct a negative signal sender as in Theorem~\ref{thm:signal_senders_existence:pendant}, take a copy of $S^+$ and a copy of $S'$ and identify the edge $e$ with one of the signal edges of $S^+$.
As a final remark, in the original manuscripts where~\ref{thm:signal_senders_existence:cycles} and~\ref{thm:signal_senders_existence:pendant} appear, it is not shown explicitly that the distance between the signal edges can be arbitrarily large. However, it is easy to see that this is indeed the case. Both constructions do guarantee that the signal edges are not incident to each other, which means that we can increase the distance between the signal edges by stringing several signal senders together (that is, taking signal senders $S_1, \dots, S_r$ and, for each $i\in \set{2,\dots, r-1}$, identifying one signal edge of $S_i$ with a signal edge of $S_{i-1}$ and the other with a signal edge of $S_{i+1}$; if we take $S_1,\dots, S_{r-1}$ to be positive signal senders, then the resulting signal sender is of the same type (positive or negative) as $S_r$).

\medskip

Indicators were introduced by Burr et al.\ in \cite{burr1977ramseyminimal} for two colors and generalized by Clemens et al.\ in \cite{clemens2018minimal} to multiple colors. Together with signal senders, these graphs will serve as basic building blocks for our pattern gadgets. For our construction, we need to modify slightly the definition appearing in \cite{clemens2018minimal}, as given below. In addition, we will need both positive and negative indicators. 

\begin{definition}\label{def:indicators}\rm
Let $q\geq 2$ and $d\geq 1$, and let $H$ and $F$ be graphs such that $H\not\subseteq F$. A \textit{positive indicator} $I = I^+ (H,F,e,q,d)$ for $H$ is a graph such that the following properties hold:
\begin{enumerate}[label=\itmarab{I}]
    \item\label{def:indicator:subgraph} $F\subseteq_{ind} I$ and $e\in E(I)$ with $\dist_I (F,e) \geq d$.
    \item \label{def:indicator:Hfree} There exists an $H$-free $q$-coloring of $I$ in which $F$ is monochromatic.
    \item\label{def:indicator:same} For every $H$-free $q$-coloring $c$ of $I$ in which $F$ is monochromatic, we have $c(e) = c(F)$.
    \item\label{def:indicator:non-constant}
    For any non-constant coloring $\varphi_F:E(F)\rightarrow [q]$ and $k\in [q]$, there exists an $H$-free coloring $c:E(I)\rightarrow [q]$ such that $c_{|F} = \varphi_F$ and $c(e) = k$.
\end{enumerate}
If $I$ is a positive indicator with parameters $H,F,e,q,$ and $d$, we call $I$ a \emph{positive $(H,F,e,q,d)$-indicator}. In this case, we call $F$ the {{\em indicator subgraph}} and $e$ the {{\em indicator edge}} of $I$.

A {\em negative indicator} $I = I^-(H,F,e,q,d)$ is the same except that in property~\ref{def:indicator:same} we replace ``$c(e) = c(F)$'' with ``$c(e) \neq c(F)$.''

 An \emph{interior} vertex of an indicator is a vertex that belongs to neither the indicator subgraph nor the indicator edge. The \emph{interior} of an indicator is the set of all interior vertices. 
\end{definition}

The construction of indicators for the case when $H$ is 3-connected or isomorphic to $K_3$ was given in~\cite{burr1977ramseyminimal} for two colors and in~\cite{clemens2018minimal} for more than two colors,
where~\ref{def:indicator:non-constant} is replaced with a similar yet slightly weaker property. Essentially the same constructions work for
3-connected graphs as well as cycles and cliques with a pendant edge
with this new property~\ref{def:indicator:non-constant}. 
In our constructions, however, we need to ensure that when we put together several gadgets and later on color each of them avoiding a monochromatic copy of our target graph $H$, there is still no monochromatic $H$ in the resulting graph. 
We do not want to accidentally create monochromatic copies that use vertices from several different pieces of our construction. While we can get this almost immediately for 3-connected graphs, in the latter two cases we need to maintain some extra properties. 
 Despite these additional technicalities and the slight modification in our definition of indicators, our proofs that the constructions given in~\cite{burr1977ramseyminimal} and~\cite{clemens2018minimal} indeed give the required positive indicators are very similar to the proofs presented in the original papers. This is why we choose to omit the proof of Theorem~\ref{thm:indicators_existence} here; for the convenience of the reader we include it in the appendix.

\begin{theorem}\label{thm:indicators_existence}
Let $q\geq 2$ and $d\geq 1$ be integers, $H$ be a graph, and $F$ be a graph with $e(F) \geq 2$.
\begin{enumerate}[label=(\alph*)]
    \item \label{thm:indicators_existence:3connected} If $H$ is $3$-connected or {$H\cong K_3$}, then a positive indicator $I=I^+(H,F,e,q,d)$ exists.
    \item \label{thm:indicators_existence:cycles} If $H\cong  C_t$ for $t\geq 4$ and $\text{girth}(F)> t$, then a positive indicator $I=I^+(H,F,e,q,d)$ with girth $t$ exists.
    \item \label{thm:indicators_existence:pendant} If $H \cong  K_t\cdot K_2$ for $t\geq 3$ and $q=2$, then there exists a positive indicator $I=I^+(H,F,e,d,q)$ with the following additional property: Both the $H$-free $2$-colorings
    in~\ref{def:indicator:Hfree} and~\ref{def:indicator:non-constant} can be chosen so that
    none of the vertices of $F$ and $e$ is a vertex of a monochromatic copy of $K_t$ and  all edges incident to $e$ have a different color from $e$.
\end{enumerate}
 Further, in parts~\ref{thm:indicators_existence:3connected} and~\ref{thm:indicators_existence:cycles} the indicators can be taken so that $(I,F)$ is $H$-robust and in part~\ref{thm:indicators_existence:pendant} we can ensure that $(I,F)$ is $K_t$-robust.
\end{theorem}

The existence of negative indicators in all of the above cases now follows immediately and is given in the following corollary.

Here and in the next sections, we will often say that we {\em join} or {\em connect} two edges $e_1, e_2$ of a given graph by a signal sender. What we mean by that is that we create a vertex-disjoint copy of a signal sender $S$ and identify its signal edges with $e_1$ and $e_2$, that is, the signal sender does not share any vertices or edges with the original graph except for the (vertices of the) signal edges. Similarly, joining or connecting a subgraph $F$ and an edge $e$ by an indicator will mean that we create a vertex-disjoint copy of the indicator and identify the indicator subgraph with $F$ and the indicator edge with $e$. We will also use the same terminology in the context of generalized negative indicators, defined later in this section. 

\begin{cor}\label{cor:neg_indicators_existence}
Let $q\geq 2$ and $d\geq 1$ be integers, $H$ be a graph, and $F$ be a graph with $e(F) \geq 2$.
\begin{enumerate}[label=(\alph*)]
    \item \label{cor:neg_indicators_existence:3connected} If $H$ is $3$-connected or {$H\cong K_3$}, then a negative indicator $I=I^-(H,F,e,q,d)$ exists.
    \item \label{cor:neg_indicators_existence:cycles} If $H\cong  C_t$ for $t\geq 4$ and $\text{girth}(F)> t$, then a negative indicator $I=I^-(H,F,e,q,d)$ with girth $t$ exists.
    \item \label{cor:neg_indicators_existence:pendant} If $H \cong  K_t\cdot K_2$ for $t\geq 3$ and $q=2$, then there exists a negative indicator $I=I^-(H,F,e,q,d)$ with the following additional property: Both the $H$-free $2$-colorings
    in~\ref{def:indicator:Hfree} and~\ref{def:indicator:non-constant} can be chosen so that
    none of the vertices of $F$ and $e$ is a vertex of a copy of monochromatic copy of $K_t$ and  all edges incident to $e$ have a different color from $e$.
\end{enumerate}
 Further, in parts~\ref{cor:neg_indicators_existence:3connected} and~\ref{cor:neg_indicators_existence:cycles}, the indicators can be taken so that $(I,F)$ is $H$-robust and in part~\ref{cor:neg_indicators_existence:pendant}, we can ensure that $(I,F)$ is $K_t$-robust.
\end{cor}
\begin{proof}
Having already established the existence of positive indicators in all the above cases, we can now construct negative indicators easily as follows:
\begin{enumerate}
    \item [(i)]  Let $I' = I^+(H,F,e',q,d)$ be a positive indicator satisfying all required additional properties.
    \item [(ii)] Let $e$ be an edge disjoint from $I'$.
    \item [(iii)] Connect $e$ and $e'$ by a negative signal sender for $H$.
\end{enumerate}

Now, properties~\ref{def:indicator:subgraph}--\ref{def:indicator:non-constant} as well as the robustness property and the additional properties required in parts~\ref{cor:neg_indicators_existence:cycles} and~\ref{cor:neg_indicators_existence:pendant} are all easy to verify.
\end{proof}

In our later proofs we will refer to all the $2$-colorings mentioned 
in Theorem~\ref{thm:signal_senders_existence}\ref{thm:signal_senders_existence:pendant},
Theorem~\ref{thm:indicators_existence}\ref{thm:indicators_existence:pendant}, and
Corollary~\ref{cor:neg_indicators_existence}\ref{cor:neg_indicators_existence:pendant} as {\em $K_t\cdot K_2$-special $2$-colorings}.

\subsection{Generalized negative indicators}\label{sec:gen_negative_indicators}
Before we can prove the existence of pattern gadgets
as stated in Theorem~\ref{thm:pattern_gadgets_existence},
we will first need to construct slightly weaker gadget graphs,
which we call generalized negative indicators.

Recall that a negative indicator $I=I^-(H,F,e,q,d)$
comes with an indicator subgraph $F$ and an indicator edge $e$ and has the following property: In any $H$-free $q$-coloring of $I$
that colors $F$ monochromatically, $e$ needs to get a color different
from that of $F$; but once $F$ is not monochromatic,
we can extend the $q$-coloring to an $H$-free $q$-coloring of $I$, independently
of which color is chosen for $e$. That is, in short, when $F$ is monochromatic we get some information on the color given to $e$,
while otherwise we do not.

The gadgets $I^\ast$ described in the following will generalize this concept
by replacing $e$ with another graph $G$. 
Now, whenever the indicator subgraph $F$ is monochromatic in an $H$-free $q$-coloring of $I^\ast$, we again want to get some information on the coloring given to $G$, namely that
a certain color pattern is forced on $G$.
Otherwise, when $F$ is not monochromatic, we do not get any information
on $G$ in the sense that we can still color this subgraph 
by any $H$-free $q$-coloring and then find an $H$-free extension to~$I^\ast$. We give a precise definition below. 

\begin{definition}\label{def:gen_negative_indicators}\rm
Let $q\geq 2$ and $d\geq 1$ be integers, and let $H, F,$ and $G$ be graphs with
$H\not\subseteq F$.
Further, let  
$G=G_1\cup G_2\cup \ldots \cup G_{q-1}$ be a partition
with $H\not\subseteq G_k$ for every $k\in [q-1]$.
We call a graph
$I^{\ast} = I^{\ast}(H,F,\{G_k\}_{k\in [q-1]},q,d)$
a \textit{generalized negative indicator} 
if the following properties hold:
\begin{enumerate}[label=\itmarab{GI}]
\item\label{def:gni:induced}
 $F,G\subseteq_{ind} I^{\ast}$ and $\dist_{I^{\ast}}(F,G)\geq d$.
\item\label{def:gni:Hfree} There exists an $H$-free $q$-coloring of $I^{\ast}$ such that $F$ is monochromatic.
\item\label{def:gni:Fmono} In any $H$-free coloring $c:E(I^\ast)\rightarrow [q]$ in which $F$ is monochromatic, 
each of the graphs $G_i$ needs to be monochromatic so that
$\{c(F),c(G_1),\ldots,c(G_{q-1})\}=[q]$.
\item\label{def:gni:Fnotmono} Let $\varphi_F:E(F)\rightarrow [q]$ be any non-constant coloring, and let $\varphi_G : E(G) \rightarrow [q]$
be any $H$-free coloring. 
Then there exists an $H$-free coloring $c:E(I^{\ast})\rightarrow [q]$
such that $c_{|F} = \varphi_F$ and $c_{|G} = \varphi_G$.
\end{enumerate}
If $I^\ast$ is a generalized negative indicator with parameters $H,F,\{G_k\}_{k\in [q-1]},q,$ and $d$, we call $I^\ast$ a \emph{generalized negative $(H,F,\{G_k\}_{k\in [q-1]},q,d)$-indicator}. In this case, 
we call $F$ and $G$ the \emph{indicator subgraphs} of $I^\ast$.

 An \emph{interior} vertex of a generalized negative indicator is a vertex that belongs to neither of the indicator subgraphs. The \emph{interior} of a generalized negative indicator is the set of all interior vertices. 
\end{definition}

The following lemma states that, if $H$ is 3-connected or isomorphic to a cycle or $K_t\cdot K_2$,
then generalized negative indicators that satisfy
additional robustness properties exist for $H$.

\begin{lemma}\label{lem:gen_negative_indicators_existence} 
Let $q\geq 2$ and $d\geq 1$ be integers, and $H, F$, and $G$ be graphs with $H\not\subseteq F$. Further, let $G=G_1\cup \ldots \cup G_{q-1}$ be a partition such that $H\not\subseteq G_k$ for every $k\in [q-1]$.
\begin{enumerate}[label = (\alph*)]
    \item\label{lem:gen_negative_indicators_existence:3connected}  If $H$ is $3$-connected or $H\cong K_3$, then a generalized negative indicator \\$I^{\ast} = I^{\ast}(H,F,\{G_k\}_{k\in [q-1]},q,d)$ exists.
    \item\label{lem:gen_negative_indicators_existence:cycles}  If $H \cong C_t$ for $t\geq 4$ and $\text{girth}(F)>t$, then a generalized negative indicator $I^{\ast} = I^{\ast}(H,F,\{G_k\}_{k\in [q-1]},q,d)$ exists.
    \item\label{lem:gen_negative_indicators_existence:pendant}  If $H\cong K_t\cdot K_2$ for $t\geq 3$ and $q=2$, then a generalized negative indicator $I^{\ast} = I^{\ast}(H,F,\{G_k\}_{k\in [q-1]},q,d)$ with the following additional property exists: Both the $H$-free $2$-colorings
    in~\ref{def:gni:Hfree} and~\ref{def:gni:Fnotmono} can be chosen so that every monochromatic copy of $K_t$ using a vertex from $F\cup G$ is contained fully in $F\cup G$.
\end{enumerate}
Further, in parts~\ref{lem:gen_negative_indicators_existence:3connected} and~\ref{lem:gen_negative_indicators_existence:cycles}, the generalized negative indicator can be taken so that $(I^{\ast},F)$ and $(I^{\ast},G)$ are $H$-robust. 
In part~\ref{lem:gen_negative_indicators_existence:pendant}, we can ensure that $(I^{\ast},F)$ and $(I^{\ast},G)$ are $K_t$-robust. 
\end{lemma}

\begin{center}
\begin{figure}[b] 
	\begin{center}
	\includegraphics[scale=0.9]{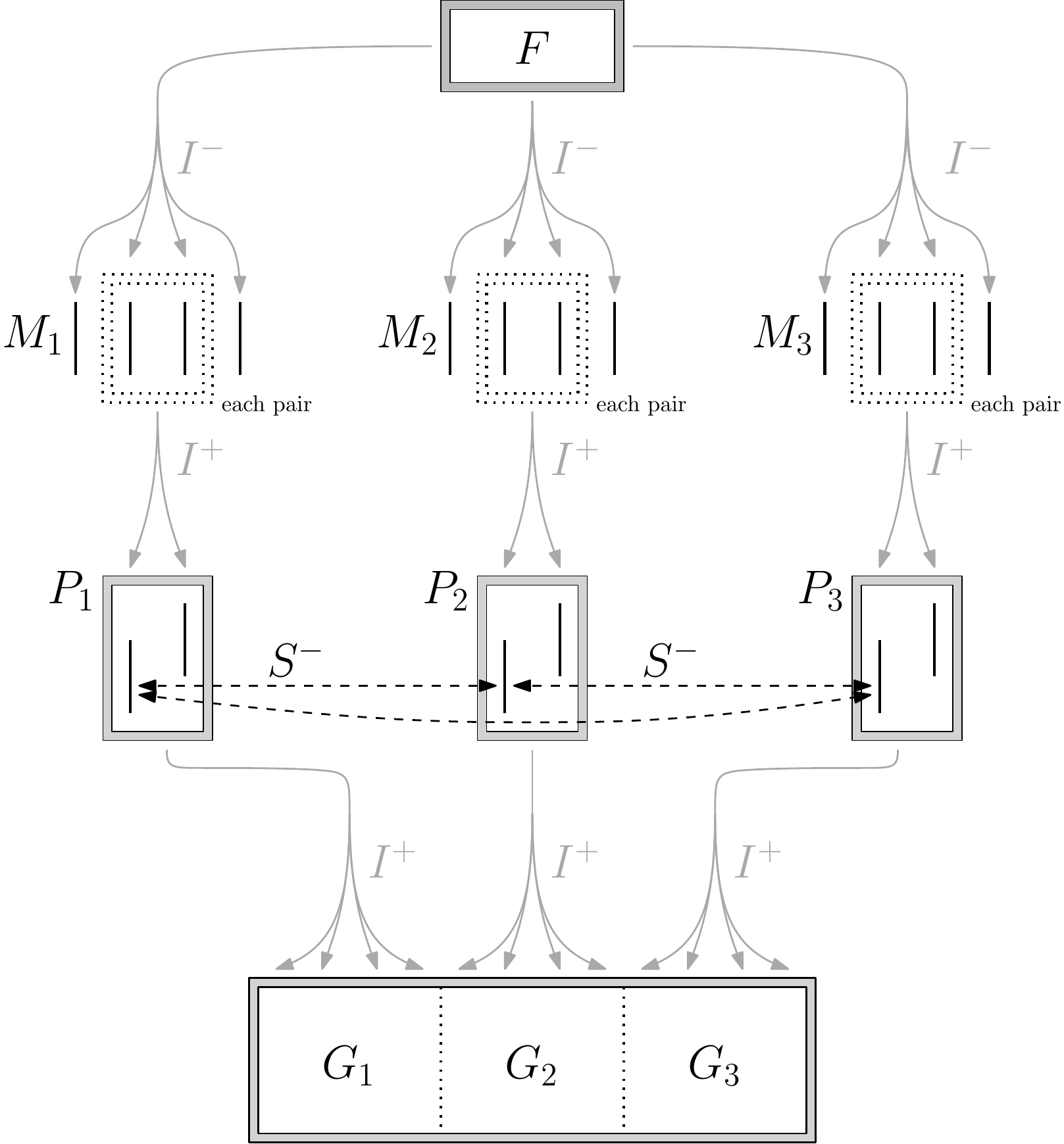}
	\end{center}
	\caption{Generalized negative indicator for $q=4$.}
	\label{fig:gnindicator}
\end{figure}
\end{center}

\begin{proof}
Let $q,d,H,F,$ and $G$ be as given, and without loss of generality assume $d\geq v(H)+1$.
Let $M_1,\ldots,M_{q-1}$ be matchings of size $q$,
let $P_1,\ldots,P_{q-1}$ be matchings of size two,
and let $e_k$ be a fixed edge of $P_k$ for each $k\in[q-1]$. 

In order to construct $I^{\ast}$, we take the vertex-disjoint union of $F$, $G$ and all of the above matchings 
and we join them with signal senders and indicators in the following way:

\begin{enumerate}
\item[(i)] For every $k\in [q-1]$ and every edge $m\in M_k$, 
	join $F$ and $m$
	by a negative $(H,F,m,q,d)$-indicator. 
\item[(ii)] For every $k\in [q-1]$, every submatching $S\subseteq M_k$ 
	of size two, and every edge $p\in P_k$,
	join $S$ and $p$ by a positive
	$(H,S,p,q,d)$-indicator.  
\item[(iii)] For every $1\leq k_1 < k_2 \leq q-1$, 
	join the distinguished edges $e_{k_1}\in P_{k_1}$ 
	and $e_{k_2}\in P_{k_2}$ 
	by a negative signal sender $S^- = S^{-}(H,e_{k_1},e_{k_2},q,d)$.  
\item[(iv)] For every $k\in [q-1]$ and every edge $g\in E(G_k)$,
	join $P_k$ and $g$ by a positive
	$(H,P_k,g,q,d)$-indicator. 
\end{enumerate}

Moreover, let all the indicators satisfy the robustness property promised by Theorem~\ref{thm:indicators_existence} and Corollary~\ref{cor:neg_indicators_existence} respectively. 
When $H$ is a cycle of length $t\geq 4$, choose the gadgets in (i)--(iv) so that their girth equals~$t$. When $H \cong K_t\cdot K_2$ for some $t\geq 3$ and $q=2$,
choose these gadgets so that they have a $K_t\cdot K_2$-special $2$-coloring.
Note that the existence of all these
gadgets and colorings is given by Theorem~\ref{thm:signal_senders_existence}, Theorem~\ref{thm:indicators_existence}, and Corollary~\ref{cor:neg_indicators_existence}.
An illustration of the construction for the case $q=4$ can be found in Figure~\ref{fig:gnindicator}.

\medskip

Let $M_k=\{m_1^k,\ldots,m_q^k\}$ for every $k \in [q-1]$. 
Before showing that $I^{\ast}$ satisfies~\ref{def:gni:induced}--\ref{def:gni:Fnotmono},
we first discuss where copies of $H$ can be located in the graph $I^\ast$. Note that from the following two observations
we immediately obtain the desired robustness properties
as stated in Lemma~\ref{lem:gen_negative_indicators_existence}.

\begin{obs}\label{obs:gni_subgraph1}
Let $H$ be $3$-connected or a cycle.
Let $I'$ be a graph obtained from $I^{\ast}$
by adding two new vertex sets $S_F$ and $S_G$
and any collection of edges
within $S_F\cup V(F)$ and within $S_G\cup V(G)$.
Then every copy of $H$ in $I'$ is fully contained in one of the indicators from (i), (ii), or (iv), in one of the signal senders from (iii), or in one of the subgraphs 
induced by $S_F\cup V(F)$ or $S_G\cup V(G)$.
\end{obs}

\textit{Proof.}
For a contradiction, assume that some copy $H'$ of $H$ in $I'$ forms a counterexample.
Consider first the case when $H'$ uses a vertex $v\in S_G\cup V(G)$. Since $H'$ is a counterexample, we have
$V(H')\not\subseteq S_G\cup V(G)$. Hence, $H'$ needs to use an interior vertex of one of the indicators in (iv); without loss of generality, assume it is an indicator $I_{P_1}^+$ joining $P_1$ with an edge of $G_1$.
We then have $\dist_{I^{\ast}}(P_1,G)\geq d>v(H')$ by property~\ref{def:indicator:subgraph} of the indicators in (iv), 
and thus,
since $H'$ is $3$-connected or a triangle or a cycle with $v(H')=\text{girth}(I_{P_1}^+)$, it follows that $H'\subseteq I_{P_1}^+$, a contradiction.
We may therefore assume that $H'$ is vertex-disjoint from $S_G\cup V(G)$.

Consider next the case when $H'$ uses a vertex $v\in S_F\cup V(F)$. 
As before, we have
$V(H')\not\subseteq S_F\cup V(F)$. 
Hence, $H'$ needs to use an interior vertex of an indicator in (i); without loss of generality, assume it is an indicator $I_1$ between $F$ and an edge $m\in M_1$.
But then, since $\dist_{I_1}(m,F)\geq d>v(H')$ by property~\ref{def:indicator:subgraph}
and since $(I_1,F)$ is $H$-robust by Theorem~\ref{thm:indicators_existence}, we conclude that $H'\subseteq I_1$
must hold, contradicting our assumption.
Hence, we may also assume that $H'$ is vertex-disjoint from $S_F\cup V(F)$.

Now, if $H'$ uses an interior vertex of one of the signal senders $S^-$ in (iii), say between the edges $e_{k_1}$ and $e_{k_2}$, then again, using that
$\dist_S(e_{k_1},e_{k_2})\geq d$ by property~\ref{def:signal_senders:dist} and 
that $H'$ is $3$-connected or isomorphic to a triangle or $H'$ is a cycle with  $v(H')=\text{girth}(S)$,
we deduce that $H'$ must be fully contained in that signal sender.

Next, if $H'$ uses an interior vertex of one of the indicators in (i), (ii) or (iv), using the same argument and the robustness properties of our indicators, guaranteed by 
Theorem~\ref{thm:indicators_existence} and Corollary~\ref{cor:neg_indicators_existence} for positive and negative indicators respectively, 
we again conclude that $H'$ must be fully contained in that indicator. 

Hence, we are left with the case when $H'$ uses neither vertices from $S_F\cup S(F)$, nor vertices from $S_G\cup S(G)$, 
nor interior vertices from one of the gadgets in (i)--(iv). But then \mbox{$H'\subseteq \bigcup_{k\in [q-1]} (M_k\cup P_k)$},
which contradicts the fact that $H'$ contains at least one cycle. \hfill { $\checkmark$}

\begin{obs}\label{obs:gni_subgraph2}
Let $H \cong K_t\cdot K_2$.
Let $I'$ be a graph obtained from $I^{\ast}$
by adding two new vertex sets $S_F$ and $S_G$
and any collection of edges
within $S_F\cup V(F)$ and within $S_G\cup V(G)$.
Then every copy of $K_t$ in $I'$ is
fully contained in one of the indicators from (i), (ii), or (iv), in one of the signal senders from (iii), or in one of the subgraphs 
induced by $S_F\cup V(F)$ or $S_G\cup V(G)$.
\end{obs}

\textit{Proof.}
The proof is analogous to the previous proof,
except that we use the robustness properties of all gadget graphs with respect to $K_t$, guaranteed by Theorem~\ref{thm:indicators_existence} and Corollary~\ref{cor:neg_indicators_existence} for the indicators
in (i), (ii), and (iv).
\hfill { $\checkmark$}

\smallskip

It remains to show that
$I^{\ast}$ satisfies~\ref{def:gni:induced}--\ref{def:gni:Fnotmono} and
to verify the additional property required in case~\ref{lem:gen_negative_indicators_existence:pendant} regarding the existence of $K_t\cdot K_2$-special $2$-colorings for~\ref{def:gni:Hfree} and~\ref{def:gni:Fnotmono}.

\medskip

\ref{def:gni:induced} The graph $F$ is an induced subgraph of $I^{\ast}$, as it is an induced subgraph of each of the negative indicators
in (i) by property~\ref{def:indicator:subgraph}.
Also $G$ is an induced subgraph of $I^\ast$, since in the construction of $I^\ast$ 
we attach gadget
graphs to single edges of $G$ without adding any further edges inside $V(G)$.
Moreover, we have $\dist_{I^{\ast}}(F,G_k)\geq d$,
since, for every $k\in [q-1]$ and every $m\in M_k$,
the joining $(H,F,m,q,d)$-indicator $I_F^{-}$ from (i) satisfies $\dist_{I_F^{-}}(F,m)\geq d$ by property~\ref{def:indicator:subgraph}.

\medskip

\ref{def:gni:Hfree} We define a coloring 
$c:E(I^{\ast})\rightarrow [q]$ as follows:
\begin{itemize}
    \item Give color 1 to the edges of $F$.
    \item For every $k\in [q-1]$, give color $k+1$ to the edges $m_1^k$ and $m_2^k$.
    \item For every $k\in [q-1]$, color the edges of $M_k\setminus \{m_1^k,m_2^k\}$ such that 
    each color from $[q]\setminus \{1,k+1\}$
    is used exactly once.
    \item For every $k\in [q-1]$, give color $k+1$ to the edges of $P_k$ and $G_k$.
    \item Finally, extend this coloring to 
    each of the indicators and signal senders in (i)--(iv) so that none of these
    contains a monochromatic copy of $H$. 
    In case~\ref{lem:gen_negative_indicators_existence:pendant}, choose these colorings to be $K_t\cdot K_2$-special. 
\end{itemize}

The extension in the last step of the coloring
is possible for the following reason:
For the indicators in (i), we can find such an extension 
by properties~\ref{def:indicator:Hfree} and~\ref{def:indicator:same} for negative indicators and
since $c(F)=1\neq c(m)$ for every $k\in [q-1]$ and $m\in M_k$. 
For the indicators in (ii), consider two cases.
If $S=\{m_1^k,m_2^k\}$, then we have $c(S)=c(P_k)=k+1$, and
hence we can color as desired by properties~\ref{def:indicator:Hfree}
and~\ref{def:indicator:same}.
Otherwise, if  $S\in \binom{M_k}{2}$ is different from $\{m_1^k,m_2^k\}$, the coloring on $S$ is not constant
and hence we can extend as desired by property~\ref{def:indicator:non-constant}.
For the signal senders in (iii), the described extension is possible by properties~\ref{def:signal_senders:Hfree} and~\ref{def:signal_senders:signal} for negative signal senders 
and since $c(e_{k_1})\neq c(e_{k_2})$ for every distinct $k_1,k_2\in [q-1]$.
For the indicators in (iv), we again use
properties~\ref{def:indicator:Hfree} and~\ref{def:indicator:same}
plus the fact that $c(P_k)=c(G_k)$ for every $k\in [q-1]$.

It remains to check that the resulting coloring $c$ on $I^{\ast}$ is $H$-free. 
Consider first the case when $H$ is a cycle or $3$-connected. By Observation~\ref{obs:gni_subgraph1}, we know that each copy of $H$ must be fully contained in one of the gadgets in (i)--(iv) or in the graph $G$. By the choice of the coloring, we know that each of the gadgets is
colored without a monochromatic copy of $H$. Moreover, the coloring $c$ splits the graph $G$ into color classes given by the subgraphs  $G_1,\ldots,G_{q-1}$, none of which contains a copy of $H$ by the assumption of the lemma. 
Hence, $c$ is $H$-free in this case.

Next, consider the case when $H\cong K_t\cdot K_2$.
Assume there is a monochromatic copy $H'$ of $H$,
and let $K'$ denote the copy of $K_t$ in $H'$.
According to Observation~\ref{obs:gni_subgraph2}, $K'$
needs to be fully contained 
in one of the gadget graphs or
in one of the subgraphs $F$ or $G$.
If $H'$ is fully contained in one of these parts, then $H'$ cannot be monochromatic by the same argument as above. Hence, we may assume that
$K'$ uses a vertex of one of the signal edges, indicator edges, or indicator subgraphs. 
If $K'$ is contained in one of the gadget graphs,
then by the choice of the $K_t\cdot K_2$-special coloring for this gadget graph,
$K'$ cannot be monochromatic, a contradiction.

So assume next that $K'\subseteq G=G_1$. 
We need to check that no edge adjacent to $K'$ can be of the same color. Indeed, since $H\not\subseteq G_1$ by the assumption of the lemma, every edge incident to $K'$ must belong to one of the indicators from (iv) and must be incident to the corresponding indicator edge which is part of $K'$. But the $2$-coloring of each indicator was chosen to be $K_t\cdot K_2$-special, so any such edge has the opposite color, and hence $H'$ cannot be monochromatic, a contradiction. 
We are left with the case $K'\subseteq F$. As we have $H\not\subseteq F$
by the assumption of the theorem, we know that any edge adjacent to $K'$ must be part of one of the indicators from (i). But then $H'$ is fully contained in such an indicator and hence cannot be monochromatic, as the coloring on every gadget
is $H$-free, a contradiction.

Note that the last argument also shows half of the additional
property in case~\ref{lem:gen_negative_indicators_existence:pendant}, i.e., that
the $H$-free $2$-colorings
    in~\ref{def:gni:Hfree} can be chosen so that
    every monochromatic copy of $K_t$ using a vertex from $F\cup G$
    is contained fully in $F\cup G$.

\medskip

\ref{def:gni:Fmono}
Let $c$ be any $H$-free coloring of $I^{\ast}$ such that $F$ is monochromatic, say $c(F)=1$. 
By properties~\ref{def:indicator:Hfree} and~\ref{def:indicator:same} for negative indicators,
the indicators in (i) make sure that all edges in the matchings $M_k$
need to get a color different from 1.
Then, by the pigeonhole principle, in each matching $M_k$ 
there needs to be at least one color from 
$[q]\setminus \{1\}$ that appears at least twice. 
For each matching $M_k$,
fix one such color and denote it by $c_k$. 
By symmetry, we assume without loss of generality that
$c(m_1^k)=c(m_2^k)=c_k$. 
By property~\ref{def:indicator:same} for the
indicators in (ii), we conclude that $c(P_k)=c(e_k)=c_k$. 
Similarly, using property~\ref{def:signal_senders:signal}
for the signal senders in (iii), we obtain that all edges in $\{e_1,\ldots,e_{q-1}\}$ need to have distinct colors.
Since color $1$ is excluded, 
we may assume by symmetry
that $c_k=c(e_k)=k+1$ and thus $c(P_k)=k+1$.
Then, applying property~\ref{def:indicator:same} for
the positive indicators in (iv) yields 
that $c(G_k)=k+1$ and hence
$\{c(F),c(G_1),\ldots,c(G_k)\}=[q]$.

\medskip

\ref{def:gni:Fnotmono}
Let $\varphi_F$ and $\varphi_G$ satisfy the assumption in property~\ref{def:gni:Fnotmono}.
We define a coloring 
$c:E(I^{\ast})\rightarrow [q]$ as follows:
\begin{itemize}
    \item Color $F$ according to $\varphi_F$.
    \item Color $G$ according to $\varphi_{G}$.
    \item For every $k\in[q-1]$ and $\ell\in [q]$, give color $\ell$ to $m_{\ell}^k$.
    \item For every $k\in [q-1]$, give color $k$
    to $e_k$ and give color $k+1$ to
    the edge in $P_k-e_k$.
    \item Finally, extend this coloring to 
    each of the indicators and signal senders in (i)--(iv) so that none of these
    contains a monochromatic copy of $H$. 
    In case~\ref{lem:gen_negative_indicators_existence:pendant}, choose these colorings to be $K_t\cdot K_2$-special. 
\end{itemize}

The extension in the last step of the coloring
is possible for the following reason:
For the indicators in (i), we can find such an extension
by property~\ref{def:indicator:non-constant} for negative indicators and since $\varphi_F$ is not constant by assumption. 
For the indicators in (ii), such an extension exists by
property~\ref{def:indicator:non-constant}
and since
no subgraph $S\subseteq M_k$ of size two is colored
monochromatically.
For the signal senders in (iii), this extension is possible by properties~\ref{def:signal_senders:Hfree} and~\ref{def:signal_senders:signal}
and since $c(e_{k_1})\neq c(e_{k_2})$ for every distinct $k_1,k_2\in [q-1]$.
For the indicators in (iv), we again use 
property~\ref{def:indicator:non-constant}
plus the fact that $P_k$ 
is not monochromatic.

Finally, as in the discussion of~\ref{def:gni:Hfree},
it follows that $c$ must be $H$-free. Moreover, if $H\cong K_t\cdot K_2$ and $q=2$
then, taking a $K_t\cdot K_2$-special $2$-coloring for each of the gadget graphs,
we deduce that every monochromatic copy of $K_t$ that uses a vertex from $F\cup G$ is fully contained  in $F\cup G$.
That is, we obtain the second half of the additional property required in
case~\ref{lem:gen_negative_indicators_existence:pendant}.
\qedhere
\end{proof}

\subsection{Existence of pattern gadgets}\label{sec:pattern_gadgets_existence}

We now prove Theorem~\ref{thm:pattern_gadgets_existence}.

Set $t=|\mathscr{G}|$. For every 
$g = \{G_1,\ldots,G_q\}\in \mathscr{G}$, fix
an \emph{ordered color pattern}
$\vec{g}=(G_1,\ldots,G_q)$
with an arbitrary ordering of the subgraphs $G_i\in g$,
and denote the $j$th component of $\vec{g}$ by $\vec{g}_j$.
Further, let $\vv{\mathscr{G}} = \left\{\vec{g}:~ g\in \mathscr{G}\right\}$.
Choose $r\in \mathbb{Z}_{\geq 1}$ such that
$$
\binom{(r-1)q+1}{r} \geq t~ .
$$
Fix a matching $M$ of size $(r-1)q+1$ and a surjection
$
s:\binom{M}{r} \rightarrow \vv{\mathscr{G}}
$,
which exists by the choice of~$r$.
We construct a pattern gadget $P=P(H,G,\mathscr{G},q)$ as follows.
Take $G$ together with the given family $\mathscr{G}$ of $H$-free 
$q$-color patterns for $G$.
Further, take the matching $M$ to be vertex-disjoint from $G$  
and join submatchings of $M$ and edges of $G$
by generalized negative indicators and positive indicators as described below. For this, choose an integer $d$ such that $d>v(H)$. 
\begin{enumerate}
\item[(i)] For every $A\in \binom{M}{r}$ and every 
	edge $e\in E( s(A)_q )$, 
	join the submatching $A$ and the edge $e$
	by a positive $(H,A,e,q,d)$-indicator. 
\item[(ii)] For every $A\in \binom{M}{r}$, 
	join the submatching $A$ and the graph $G-(s(A))_q$
	by a generalized negative $(H,A,\{ s(A)_k \}_{k\in [q-1]} ,q,d)$-indicator. 
\end{enumerate}

The existence of the indicators needed in (i) and (ii) is given by
Theorem~\ref{thm:indicators_existence} and Lemma~\ref{lem:gen_negative_indicators_existence}.

In the case when $H\cong K_t\cdot K_2$ and $q=2$, 
we additionally choose all gadgets so that they
have $K_t\cdot K_2$-special $2$-colorings
as described in Theorem~\ref{thm:indicators_existence} and Lemma~\ref{lem:gen_negative_indicators_existence}\ref{lem:gen_negative_indicators_existence:pendant} respectively.
Moreover, we choose all the indicators so that they satisfy the robustness properties described in Theorem~\ref{thm:indicators_existence} and Lemma~\ref{lem:gen_negative_indicators_existence}.
Then, analogously to Observation~\ref{obs:gni_subgraph1} and Observation~\ref{obs:gni_subgraph2}, we can prove the following.

\begin{obs}\label{obs:pg_subgraph}
Let $P'$ be a graph obtained from $P$
by adding a vertex set $S$
and any collection of edges
within $S\cup V(G)$.
If $H$ is $3$-connected or a cycle,
then every copy of $H$ in $P'$ is
fully contained in one of the indicators from (i) or (ii)
or in the subgraph induced by $S\cup V(G)$.
If $H\cong K_t\cdot K_2$,
then every copy of $K_t$ in $P'$ is
fully contained in one of the indicators from (i) or (ii)
or in the subgraph induced by $S\cup V(G)$.
\end{obs}

Given this observation, it follows immediately
that $(P,G)$ is $H$-robust if $H$ is 3-connected or a cycle
and that $(P,G)$ is $K_t$-robust if $H\cong K_t\cdot K_2$.
Hence, it remains to verify that $P$ satisfies ~\ref{def:patterngadget_subgraphG}--\ref{def:patterngadget_fixedpattern},
and that in the case when $H\cong K_t\cdot K_2$ and $q=2$
we can find $2$-colorings for~\ref{def:patterngadget_fixedpattern}
as described in part \ref{thm:pattern_gadgets_existence:pendant} of Theorem~\ref{thm:pattern_gadgets_existence}.

\medskip

\ref{def:patterngadget_subgraphG}
Since $P$ is constructed by attaching different gadgets to $G$ without adding edges inside $V(G)$, we have
$G\subseteq_{ind} P$.

\medskip

\ref{def:patterngadget_somepattern} 
Let $c:E(P)\rightarrow [q]$ be any $H$-free coloring of $P$. 
By the pigeonhole principle, at least one color is used at least 
$r$ times on the matching $M$. Without loss of generality, say $c(A)=q$ for some $A\in \binom{M}{r}$.
Consider the pattern $g=\{s(A)_k\}_{k\in [q]}$.  
By property~\ref{def:indicator:same}
of the indicators in (i), we deduce that
every edge in $E(s(A)_q)$ also needs to have color $q$. 
Moreover, 
by property~\ref{def:gni:Fmono} of the generalized 
negative indicators in (ii),  each of the subgraphs $s(A)_k$ with $k\neq q$ 
is forced to be monochromatic, and all colors except for $c(A)=q$ get used among these subgraphs. 
Hence, $\{c_{|G}^{-1}(1), \dots, c_{|G}^{-1}(q)\} = g
\in \mathscr{G}$.

\medskip

\ref{def:patterngadget_fixedpattern} 
Let $g =\{G_1,\ldots,G_q\}\in \mathscr{G}$
be given. Fix an arbitrary set $A_0\in\binom{M}{r}$ such
that $s(A_0)=\vec{g}$. Without loss of generality, assume $s(A_0)_k=G_k$
for every $k\in [q]$; otherwise relabel the subgraphs in $g$.
We define a coloring $c:E(P)\rightarrow [q]$
as follows:
\begin{itemize}
    \item Give color $q$ to each edge in $A_0$.
    \item Color $M\setminus A_0$ so that each
    color from $[q-1]$ appears exactly $r-1$ times.
    \item For every $k\in [q]$, give color $k$
    to the edges of $G_k$.
    \item Finally, 
    extend this coloring to 
    each of the gadgets in (i) and (ii) 
    so that none of these
    contains a monochromatic copy of $H$. 
    In case~\ref{lem:gen_negative_indicators_existence:pendant}, choose these colorings to be $K_t\cdot K_2$-special. 
\end{itemize}

We claim that the extension in the last step of the coloring is indeed possible. Recall that each gadget from (i) and (ii) is associated to a submatching $A \in \binom{M}{r}$.
Suppose first that $A=A_0$. Then we have $c(A)=q=c(s(A)_q)$, and by properties~\ref{def:indicator:Hfree} and~\ref{def:indicator:same}, 
we find an extension as desired for the corresponding positive indicators in (i). 
Moreover, we have $c(s(A)_k)=k\neq q = c(A)$ for every color 
$k\in [q-1]$. Hence, by properties~\ref{def:gni:Hfree} and~\ref{def:gni:Fmono},
we find extensions as desired for the corresponding generalized negative indicators in (ii).
Consider next the case when $A\neq A_0$. Then $A$ is not monochromatic, since $A_0$ is the only monochromatic subset of $M$ of size $r$. 
Now, let $I$ be any positive indicator between $A$ and any edge
$e\in E( s(A)_q )$ as described in (i). Then, by property~\ref{def:indicator:non-constant}, we 
find an extension for $I$ as desired. 
Finally, let $I$ be the generalized negative indicator from (ii) 
for the set $A$. Then, using property~\ref{def:gni:Fnotmono}, 
we conclude analogously that an extension for $I$ can be found.

Finally, we have $\{G_1,\ldots,G_q\}=
\{c_{|G}^{-1}(1),\ldots,c_{|G}^{-1}(q)\}$. Since
$g=\{G_1,\ldots,G_q\}$ is an $H$-free $q$-color pattern
by the assumption of the theorem, 
we know that $c_{|G}$ is $H$-free. 
Now, if $H$ is 3-connected or a cycle, 
then every copy of $H$ in $P$ that is not contained in $G$
must be a subgraph of some indicator from (i) or (ii),
according to Observation~\ref{obs:pg_subgraph}.
But we already know that the coloring $c$ is $H$-free on every
indicator, and hence it is $H$-free on the whole graph $P$.

It remains to consider the case when
$H\cong K_t\cdot K_2$ and $q=2$. Assume there is a monochromatic copy $H'$ of $H$,
and let $K'$ denote its copy of $K_t$.
As above, if $H'$ is fully contained in one of the indicators, then it cannot be monochromatic.
Hence, we may assume that
$K'$ intersects the vertex set of an indicator edge or an indicator subgraph.
Then, by the 
$K_t\cdot K_2$-special $2$-colorings for the indicators,
we know that $K'$ needs to be
a subgraph of $G$. Without loss of generality, let $K'\subseteq G_1$.
Since $G$ does not contain a copy of $K_t\cdot K_2$ by assumption, we know that $E_G(V(K'),V(G_2))=\varnothing$.
Hence, the pendant edge $f$ of $H'$ needs
to belong either to a positive indicator between some $A\in \binom{M}{r}$ and some $e\in E(K')$, or to a generalized
negative indicator between some $A\in\binom{M}{r}$ and the graph $G_1 \supseteq K'$. In the former case, the edge $f$ needs to be incident to the indicator edge $e$ and hence
$c(e)\neq c(f)$ by the $K_t\cdot K_2$-special $2$-coloring of the corresponding positive indicator. In the latter case,
we have $c(f)\neq c(K')$ as the coloring of the generalized negative indicator was chosen to be $H$-free. Hence, in both cases $H'$ cannot be monochromatic, a contradiction.
\hfill $\Box$

\section{Applications of pattern gadgets}\label{sec:applications}

In this section, we present several applications of the pattern gadgets constructed in the previous section. We first prove Theorem~\ref{cor:arbitrarily_many_cycles}
and Theorem~\ref{cor:arbitrarily_many_Kt.k2} directly. The proof of Theorem~\ref{cor:arbitrarily_many_cliques} is given later in the section as a consequence of a more general result about 3-connected graphs (Theorem~\ref{thm:arbitrarily_many_3-connected}). 

\begin{proof}[Proof of Theorem~\ref{cor:arbitrarily_many_cycles}]
Let $H\cong C_t$ and $t \geq 4 $ and $q\geq 2$ be fixed. 
We first note that $s_q(H) \geq q+1$. Indeed, suppose there is a minimal $q$-Ramsey graph $G$ for $C_t$ with a vertex $v$ of degree at most $q$; by the minimality of $G$, there exists a $C_t$-free $q$-coloring of $G-v$. Now, coloring the edges incident to $v$ so that no two of them share a color gives a $q$-coloring of $G$ with no monochromatic $C_t$, a contradiction.

\begin{center}
\begin{figure}[t] 
	\begin{center}
	\includegraphics[scale=0.9]{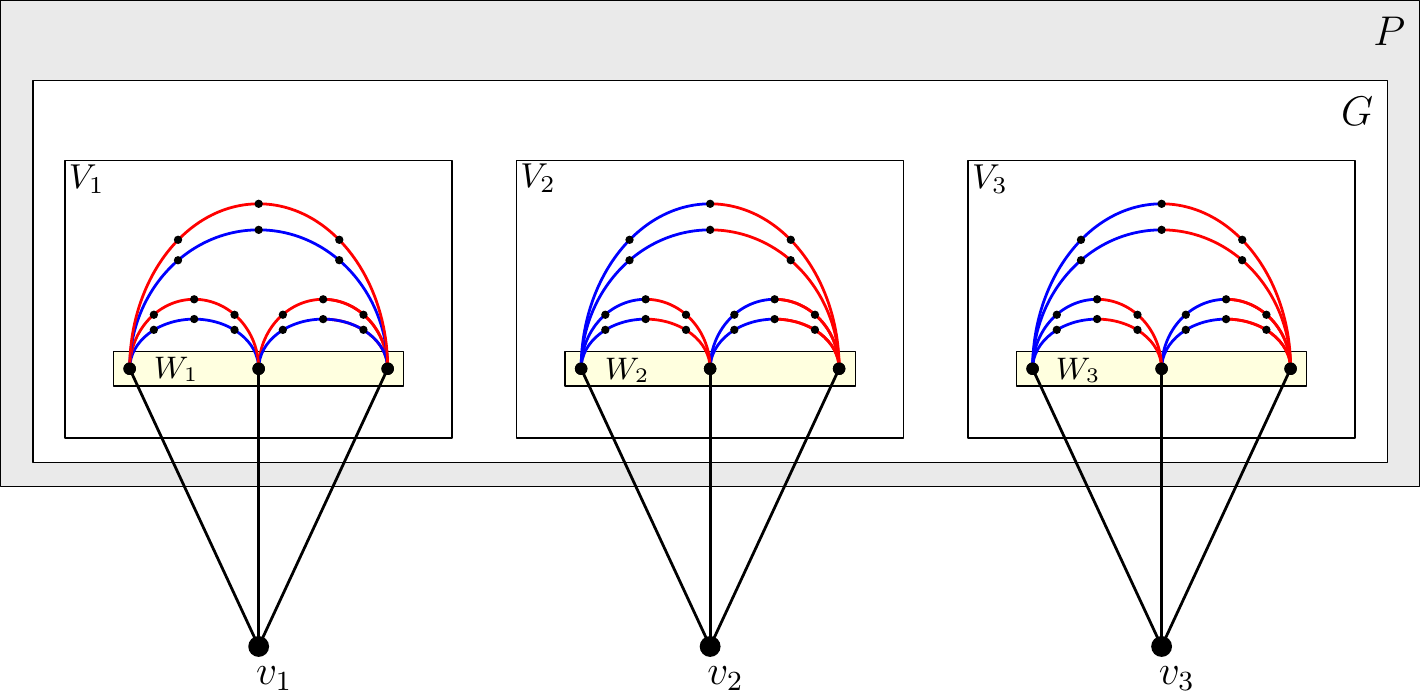}
	\end{center}
	\caption{Graph $\widetilde{G}$ for $q=2$, $t=6$ and $k=3$.}
	\label{fig:cycle-pattern}
\end{figure}
\end{center}

We now turn our attention to showing that there can be arbitrarily many vertices of degree $q+1$, also implying  that $s_q(C_t) = q+1$. Let $k\geq 1$. We now construct a minimal $q$-Ramsey graph for $H$ with at least $k$ vertices of degree $q+1$. %This also proves that $s_q(H)= q+1$.
%For the ease of representation, the figure shows the case when $q=2$, $t=6$ and $k=3$. 
Our graph will be constructed in several steps.
We refer the reader to Figure~\ref{fig:cycle-pattern} for an illustration of our construction in the case when $q=2$, $t=6$, and $k=3$. 

To begin with, let $W$  be a set of $q+1$ vertices.
For every $u,w \in W$ and $u \neq w$, add %connect $u$ and $w$ with 
$q$ internally vertex-disjoint paths of length $t-2$ with $u$ and $w$ as endpoints. Call the resulting graph $F$. Let $c_1: E(F) \rightarrow [q]$ be a coloring of the edges of $F$ such that, for every distinct $u,w \in W$, every path between $u$ and $w$ is monochromatic but no two such paths are monochromatic in the same color. Let $c_2: E(F) \rightarrow [q]$ be another coloring of the edges of $F$ such that, for every distinct $u,w \in W$, no path between $u$ and $w$ is monochromatic. We define $f_1$ and $f_2$, two $q$-color patterns for $F$, by setting $f_1=\{c_1^{-1}(i)\}_{i \in [q]}$ and $f_2=\{c_2^{-1}(i)\}_{i \in [q]}$. Note that $f_1$ and $f_2$ are $H$-free.

We now take $k$ vertex-disjoint copies $F_1,\dots, F_k$ of $F$, where $F_i=(V_i,E_i)$ for all $1\leq i\leq k$, and
denote by $W_i$ the subset of $V_i$ corresponding to $W$ in $V(F)$. Call this graph $G$, and define  $V = \bigcup\limits_{i=1}^k V_i$. Note that $G\not\rightarrow_{q} H$, since $F\not\rightarrow_{q} H$. Let $\mathscr{G}$ be a family of $q$-color patterns for $G$ such that $g\in \mathscr{G}$ if and only if there exists an $i \in [k]$ such that $g[V_i] \cong f_1$ and $g[V_j] \cong f_2$ for all $j \neq i$. Note that $\mathscr{G}$ is a family of $H$-free $q$-color patterns for $G$.

By Theorem~\ref{thm:pattern_gadgets_existence}, we know that there exists a pattern gadget $P=P(H,G,\mathscr{G},q)$. Moreover, we can choose the pattern gadget $P$ in such a way that the pair $(P,G)$ is $H$-robust. We add $k$ additional vertices $v_1, \dots , v_k$ to $P$, and for all $i \in [k]$, we add edges from $v_i$ to all vertices in $W_i$. We call the resulting graph $\widetilde{G}$.
\smallskip

We now show that $\widetilde{G} \rightarrow_q H$ and that each of the new vertices $v_i$ is important for $\widetilde{G}$ to have this property, that is, $\widetilde{G} - v_i \not\rightarrow_q H$ for every $i \in [k]$. This then implies the existence of a minimal $q$-Ramsey graph for $H$ with the desired properties. Indeed, consider any minimal $q$-Ramsey graph $\widetilde{G}' \subseteq \widetilde{G}$. Since $\widetilde{G} - v_i \not\rightarrow_q H$, we know that $v_i \in V(\widetilde{G}')$ for every $i \in [k]$. Also $q+1 \leq s_q(C_t)   \leq d_{\widetilde{G}'}(v_i) \leq q+1$, which means that $d_{\widetilde{G}'}(v_i)=s_q(C_t)=q+1$.
\medskip

First, we show that $\widetilde{G} \rightarrow_{q} H$. Let $c: E(\widetilde{G})\rightarrow [q]$ be a $q$-coloring of the edges of $\widetilde{G}$, and assume $c$ is $H$-free. For each $i\in[q]$, define $c_i=c^{-1}(i)$ to be the $i$th color class with respect to $c$. By property~\ref{def:patterngadget_somepattern} of the pattern gadget $P$, we know that $g=\{c_1[V], \dots , c_q[V]\} \in \mathscr{G}$; by the definition of $\mathscr{G}$, there exists an $i \in [k]$ such that $\{c_1[V_i], \dots , c_q[V_i] \}\cong f_1$. Without loss of generality, we may assume $i=1$. Consider the edges from $v_1$ to the vertices of $W_1$. There are $q+1$ such edges and they are colored in $q$ colors, so by the pigeonhole principle there are two vertices in $W_1$, say $u$ and $w$, such that $c(v_1u)=c(v_1w)$. Again without loss of generality, we may assume $c(v_1u)=1$. By our choice of $f_1$, we know that there is a monochromatic path of length $t-2$ in color $1$ between the vertices $u$ and $w$. This monochromatic path along with the edges $v_1u$ and $v_1w$ gives a monochromatic cycle of length $t$, contradicting our assumption. 
\smallskip

Next, we show  that $\widetilde{G} - v_i \not\rightarrow_{q} H$ for every $i \in [k]$. By symmetry, it is enough to show this for $i =1$. Partition the vertices in $V$ in the following way: For every $\ell \in [k]$, write $G[V_\ell]= G_{\ell,1} \cup \dots \cup G_{\ell,q}$ so that $\{G_{1,j}\}_{j \leq q} \cong f_1$ and $\{G_{\ell,j}\}_{j \leq q} \cong f_2$ for $\ell \neq 1$. We define a coloring $c: E(G)\rightarrow [q]$ by setting $c(G_{\ell,j})=j$ for every $\ell \in [k]$ and $j \in [q]$. The $q$-color pattern on $V$ defined by $c$, namely $\set{c_{|G}^{-1}(1),\dots, c_{|G}^{-1}(q)}$, is in $\mathscr{G}$, and by property~\ref{def:patterngadget_fixedpattern}, we can extend $c$ to an $H$-free coloring of $P$. We then color the remaining edges in $\widetilde{G} - v_1$ arbitrarily, and denote the resulting $q$-coloring of $\widetilde{G}-v_1$ by $\Tilde{c}$. Since $\Tilde{c}_{|P}$ is $H$-free, any monochromatic copy of $H$ in $\widetilde{G}- v_1$ needs to contain a vertex $v_\ell$ for some $\ell \geq 2$. Now, due to the $C_t$-robustness of the pair $(P,G)$, any possible monochromatic copy of $H$ must be contained in some $V_\ell \cup \{v_\ell\}$. Such a copy then needs to contain two vertices of $W_\ell$ and a path of length $t-2$ between them. But we know that $\{c_{|G[V_\ell]}^{-1}(j)\}_{j \in [q]} \cong f_2$, and by the definition of $f_2$, no such path is monochromatic. Hence, no monochromatic copy of $H$ exists. 
\end{proof}

\begin{proof}[Proof of Theorem~\ref{cor:arbitrarily_many_Kt.k2}]
It was shown by Fox et al.~\cite{fox2014ramsey} that $s_2(K_t \cdot K_2)=t-1$ for every $t \geq 3$. We now show that a minimal $2$-Ramsey graph for $K_t \cdot K_2$ can contain arbitrarily many vertices of this minimum degree.

Let $H\cong K_t \cdot K_2$ for some $t \geq 3$, and let $k\geq 1$ be fixed.
Our construction of a minimal $2$-Ramsey graph for $H$ containing at least $k$ vertices of degree $t-1$ will combine ideas similar to those in the proof of Theorem~\ref{cor:arbitrarily_many_cycles} with ideas from the construction given by Fox et al.~\cite{fox2014ramsey}.  
We again refer the reader to Figure~\ref{fig:ktk2-pattern} for an illustration of the case $t=4$ and $k=3$. 

We begin by defining $F$ to be the vertex disjoint union of $t-1$ copies of $K_t$. For every copy of $K_t$, we fix an arbitrary vertex and call the set of all these vertices $W$. Let $c_1: E(F) \rightarrow \{red, blue\}$ be a $2$-coloring that colors every edge of $F$ red. 
%of the edges of $F$ such that for every $e \in E(F)$, $c_1 (e)= red$. 
Let $c_2: E(F) \rightarrow \{red, blue\}$ be another $2$-coloring of the edges of $F$ such that no copy of $K_t$ is monochromatic (in either color). We define two color patterns $f_1$ and $f_2$ for $F$ by setting $f_1=\{c_1^{-1}(red), c_1^{-1}(blue)\}$ and $f_2=\{c_2^{-1}(red), c_2^{-1}(blue)\}$. Note that $f_1$ and $f_2$ are $H$-free. 

Now take $k$ vertex-disjoint copies $F_1,\dots, F_k$ of $F$, where $F_i=(V_i,E_i)$ for $1\leq i\leq k$, and let $W_i$ be the subset of $V_i$ corresponding to the set $W$ in $V(F)$. Call this graph $G$, and define $V = \bigcup\limits_{i=1}^k V_i$. Note that $G$ does not contain any copies of $H$. Let $\mathscr{G}$ be a family of $2$-color patterns for $G$ such that $g \in \mathscr{G}$ if and only if there exists an $i \in [k]$ such that $g[V_i] \cong f_1$ and $g[V_j] \cong f_2$ for all $j \neq i$. Note that $\mathscr{G}$ is a family of $H$-free $2$-color patterns for $G$.

By Theorem~\ref{thm:pattern_gadgets_existence}, we deduce that there exists a pattern gadget $P=P(H,G,\mathscr{G},2)$. Moreover, we can choose the pattern gadget $P$ in such a way that the pair $(P,G)$ is $K_t$-robust and that for property~\ref{def:patterngadget_fixedpattern} there is always an $H$-free $2$-coloring such that, if a monochromatic copy of $K_t$ uses a vertex from $G$, then it lies entirely in $G$. We add $k$ additional vertices $v_1, \dots , v_k$ to $P$ with edges from $v_i$ to all vertices of $W_i$ for all $i \in [k]$; also, for all $i\in[k]$, we add an edge between each pair of distinct vertices in $W_i$.
Lastly, we choose an arbitrary vertex in $W_i$ and add a pendant edge $e_i$ incident to that vertex. We call the resulting graph $\widetilde{G}$.

\begin{center}
\begin{figure}[t] 
	\begin{center}
	\includegraphics[scale=0.9]{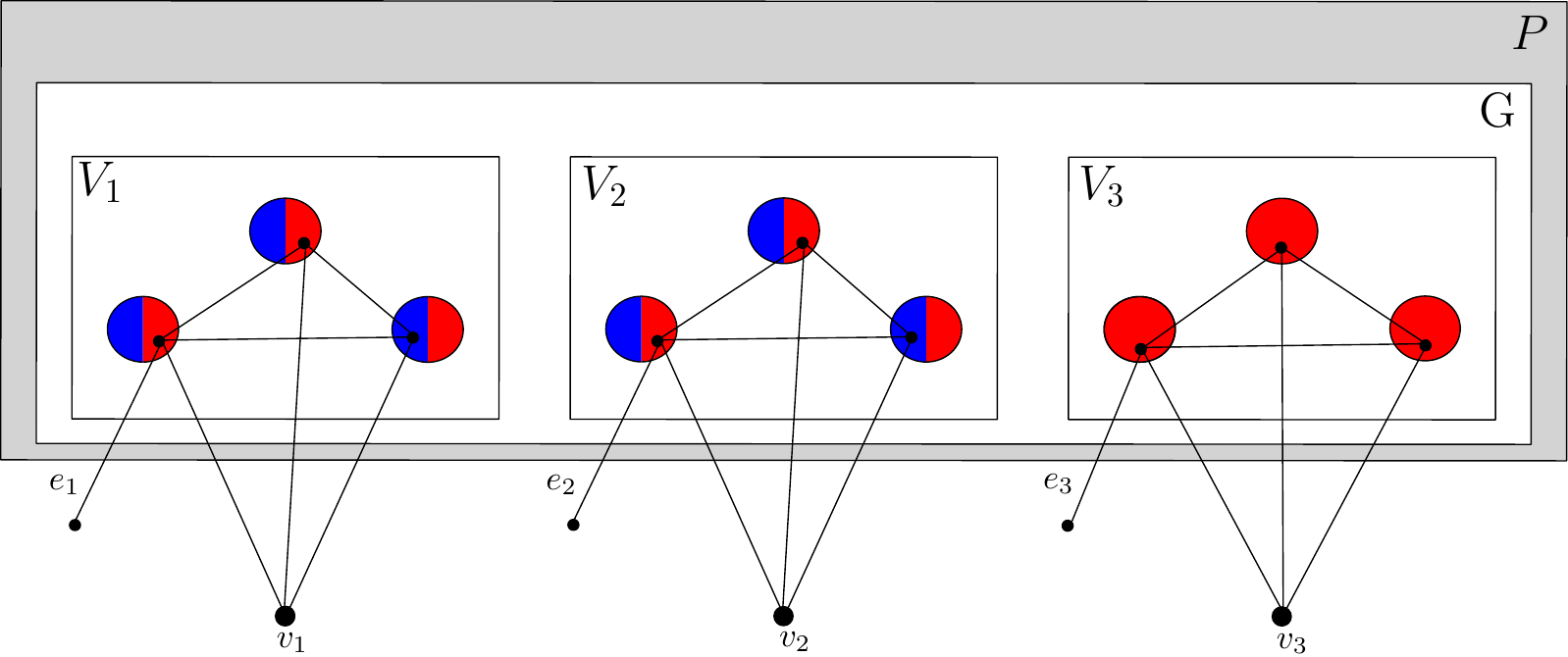}
	\end{center}
	\caption{Graph $\widetilde{G}$ for $t=4$ and $k=3$.}
	\label{fig:ktk2-pattern}
\end{figure}
\end{center}

We now show that $\widetilde{G} \rightarrow_2 H$ and that $\widetilde{G} - v_i \not\rightarrow_2 H$ for every $i \in [k]$. This, as argued in the proof of Theorem~\ref{cor:arbitrarily_many_cycles}, implies the existence of a minimal $2$-Ramsey graph with the desired properties. 

First we show that $\widetilde{G} \rightarrow_{2} H$. Let $c: E(\widetilde{G})\rightarrow \{red, blue\}$ be a $2$-coloring of the edges of $\widetilde{G}$; assume $c$ is $H$-free. Define $c_{red}=c^{-1}(red)$ and $c_{blue}=c^{-1}(blue)$ to be the two color classes with respect to $c$. By property~\ref{def:patterngadget_somepattern} of the pattern gadget $P$, we know that $g=\{c_{red}[V], c_{blue}[V]\} \in \mathscr{G}$, and by the definition of $\mathscr{G}$, there exists an $i \in [k]$ such that $\{c_{red}[V_i], c_{blue}[V_i] \}\cong f_1$. Without loss of generality, we may assume $i=1$ and that every edge inside $V_i$ is red. Consider the edges with endpoints in the set $W'=W_1 \cup \{v_1\}$. Since $c$ is an $H$-free coloring of $\widetilde{G}$ and each such edge $e$ has at least one endpoint in $W_1$ (and is hence incident to an all-red copy of $K_t$), we obtain that $c(e)=blue$. As a result, the graph induced by $W'$ is a monochromatic blue copy of $K_t$. Now, the pendant edge $e_1$ is incident to monochromatic copies of $K_t$ in both colors and thus creates a monochromatic copy of $H$ irrespective of its color. This contradicts our assumption.

Next, we show that, for every $i \in [k]$, we have $\widetilde{G} - v_i \not\rightarrow_{2} H$. By symmetry, it suffices to show this for~$i=1$. 
For every $\ell \in[k]$, take a partition $G[V_\ell]= G_{\ell,red} \cup G_{\ell,blue}$ such that $\{G_{1,red}, G_{1,blue}\} \cong f_1$ and $\{G_{\ell,red}, G_{\ell,blue}\} \cong f_2$ for $\ell \neq 1$.
We define a coloring ${c}: E(\widetilde{G})\rightarrow \set{red,blue}$ by first setting ${c}(G_{\ell,j})=j$ for every $\ell \in [k]$ and $j \in \{red, blue\}$. The color pattern defined on $G$ by ${c}$ is in $\mathscr{G}$, and by property~\ref{def:patterngadget_fixedpattern} of $P$, we can extend this to all of $P$ so that the coloring ${c}_{|P}$ is $H$-free and has the following additional property:
\begin{enumerate}
    \item[(P)] If a monochromatic copy of $K_t$ in the coloring ${c}_{|P}$ uses a vertex from $G$, then it lies entirely in $G$.
\end{enumerate}
Now, 
for every $\ell \geq 2$, color one edge between $v_\ell$ and $W_\ell$ red and color the remaining edges in $E_{\widetilde{G}}(W_\ell\cup \{v_\ell\})\cup \{e_\ell\}$ blue. Further, color all edges in $E_{\widetilde{G}}(W_1)\cup \{e_1\}$ with the color
not used on $G[V_1]$ (recall that $G[V_1]$ was colored monochromatically as
$\{G_{1,red}, G_{1,blue}\} \cong f_1$).

We claim that this coloring is $H$-free. For a contradiction, assume that 
there is a monochromatic copy $H'$ of $H$ produced by the coloring ${c}$.
Since ${c}_{|P}$ is $H$-free, $H'$ needs to use at least one edge $e_0$ from $E_{\widetilde{G}}(W_1) \cup \{e_1\}$ or
from $E_{\widetilde{G}}(W_\ell \cup \{v_\ell\}) \cup \{e_\ell\}$ for some $\ell\geq 2$.

Consider first the case when
$e_0\in E_{\widetilde{G}}(W_1) \cup \{e_1\}$. We know that $G[V_1]$
is monochromatic and that $e_0$ has the opposite color, say $G[V_1]$ is red and $e_0$ is blue. Then, by property (P) and the fact that $|W_1|=t-1$, there can be no blue copy of $K_t$ in the subgraph induced by the set $V_1\supseteq W_1$. 
But this means that $e_0$ cannot be part of a blue copy of $K_t\cdot K_2$, a contradiction.

Consider now the case when
$e_0\in E_{\widetilde{G}}(W_\ell \cup \{v_\ell\}) \cup \{e_\ell\}$ for some $\ell\geq 2$, and  assume  without loss of generality that $\ell=2$.
By the $K_t$-robustness of the pair $(P,G)$, the copy $H'$ of $H$ must be contained within $E_{\widetilde{G}}(V_2 \cup \{v_2\}) \cup \{e_2\}$. 
Since ${c}_{|G[V_2]} \cong f_2$, i.e.,
the copies of $K_t$ in $F_2$ are not monochromatic, and ${c}$ satisfies property (P), we obtain that $G[V_2]$ does not contain a monochromatic copy of $K_t$.
From this and the fact that $e_2$ is a pendant edge it follows that the vertices of the copy of $K_t$ in $H'$ must be contained entirely in $W_2\cup \set{v_2}$. But this set contains precisely $t$ vertices that do not form a monochromatic copy of $K_t$, again giving a contradiction.
\end{proof}

Before turning to the proof of Theorem~\ref{cor:arbitrarily_many_cliques},
we state and prove a more general statement concerning $3$-connected graphs.
Roughly speaking, it reduces the problem of showing
$s_q$-abundance to that of finding a suitable minimal $q$-Ramsey graph containing at least one vertex of the desired small degree. In fact, we can even relax the condition that the $q$-Ramsey graph be minimal and that the desired small degree be precisely $s_q(H)$ for the given graph $H$.

\begin{theorem}\label{thm:arbitrarily_many_3-connected}
Let $H$ be 3-connected or a triangle and assume there exists a graph $F$ together with a vertex $v\in V(F)$ and an edge $e\in E(F)$ satisfying the following properties:
\begin{enumerate}[label=\itmarab{F}]
    \item\label{goodF:Ramsey} $F\rightarrow_{q} H$.
    \item\label{goodF:shareH} $v$ and $e$ do not share a copy of $H$ in $F$.
    \item\label{goodF:notRamsey1} $F-e\not\rightarrow_q H$. \item\label{goodF:notRamsey2}
    $F-g\not\rightarrow_q H$ for every $g\in E(F)$ which is incident to $v$.
\end{enumerate}
Then, for any $k\in \mathbb{Z}_{\geq 1}$, there exists a minimal $q$-Ramsey graph for $H$ that has $k$ vertices of degree $d_F(v)$.
\end{theorem}

\begin{proof}
Given a graph $F$ with the required properties, 
denote  the edges incident to $v$ in $F$ by $g_1,\ldots,g_{d_{F}(v)}$. Let $F'=F-v-e$. In order to define $q$-color patterns for an application of Theorem~\ref{thm:pattern_gadgets_existence}, we first observe the existence of
two types of $H$-free $q$-colorings on $F'$.

\begin{clm}
For every $j\in [d_F(v)]$,
there exists an $H$-free $q$-coloring $c_{1,j}$ of $F'$ such that
\begin{itemize}
\item $c_{1,j}$ can be extended to an $H$-free $q$-coloring of $F-\{e,g_j\}$, and
\item $c_{1,j}$ cannot be extended to an $H$-free $q$-coloring of $F-e$.
\end{itemize}
\end{clm}

\textit{Proof.}
By property~\ref{goodF:notRamsey2}, there exists an $H$-free $q$-coloring $\varphi$ of $F-g_j$. We set $c_{1,j}:=\varphi_{|F'}$.
One observes easily that this is an $H$-free $q$-coloring of $F'$
and that $\varphi_{|F-\{e,g_j\}}$ is an extension to $F-\{e,g_j\}$
that is $H$-free. Hence, it remains to check that there is no $H$-free extension to the graph $F-e$. 

For a contradiction, assume that there exists some $H$-free coloring $\psi: E(F-e) \rightarrow [q]$ extending $c_{1,j}$. The $q$-coloring $\widetilde{\psi}:E(F)\rightarrow [q]$
defined by
\begin{align*}
\widetilde{\psi}(f)=
\begin{cases}
\psi(f) & \text{ if }f\neq e\\
\varphi(e) & \text{ if }f=e
\end{cases}
\end{align*}
cannot be $H$-free by property~\ref{goodF:Ramsey}. 
Thus, there must be a copy $H'$ of $H$ that is monochromatic under $\widetilde{\psi}$; moreover, $H'$ needs to use the edge $e$ as $\widetilde{\psi}_{|F-e}=\psi$ is $H$-free.
By property~\ref{goodF:shareH},
we have $v\notin V(H')$, that is, $H'$ lies entirely in the graph
$F-v$. However, $\widetilde{\psi}_{|F-v}=\varphi_{|F-v}$, since
$\widetilde{\psi}_{|F'}=\psi_{|F'}=c_{1,j}=\varphi_{|F'}$
and $\widetilde{\psi}(e)=\varphi(e)$. Hence, since $\varphi$ is $H$-free, $H'$ cannot be monochromatic, a contradiction. \hfill $\checkmark$

\begin{clm}
There exists an $H$-free $q$-coloring $c_2$ of $F'$ that can be extended to an $H$-free $q$-coloring of $F-e$. 
\end{clm}

\textit{Proof.} By property~\ref{goodF:notRamsey1} there exists an
$H$-free $q$ coloring $\varphi$ of $F-e$. We set $c_2:=\varphi_{|F'}$. \hfill \checkmark

\medskip

Given the colorings of our previous claims,
we next define $H$-free $q$-color patterns $f_{1,j}$, with $j\in [d_F(v)]$, and $f_2$ for $F'$ by partitioning $F'$ into its color classes with respect to $c_{1,j}$ and $c_2$, respectively. More precisely, we set
\begin{align*}
f_{1,j}   =\{c_{1,j}^{-1}(i)\}_{i\in [q]} ~~ \text{ and } ~~
f_2  =\{c_2^{-1}(i)\}_{i\in [q]}.
\end{align*}

Now let $k\geq 1$ be an integer. We proceed similarly as in the proof of Theorem~\ref{cor:arbitrarily_many_cycles} and construct a graph $\widetilde{G}$ that will be a $q$-Ramsey graph for $H$
with the additional property that there are at least $k$
vertices of degree $d_F(v)$, each of which is important for $G$ to be $q$-Ramsey for $H$.
 
First, let $F_1,\ldots,F_q$ be $k$ vertex-disjoint copies of $F-e$. 
For each $i\in [k]$, let $v_i\in V(F_i)$ represent 
the vertex $v\in V(F-e)$ and let $g_1^i,\ldots,g_{d_F(v)}^i\in E(F_i)$
be the edges representing $g_1,\ldots,g_{d_F(v)}$.
Moreover, for every $i\in [k]$, let $F_i'=F_i-v_i$ 
and $W_i:=N_{F_i}(v_i)$.

We fix $G=(V,E)$ to be the vertex-disjoint union of the graphs $F_i'=(V_i',E_i')$, i.e., we set $V=\cup_{i=1}^k V_i'$
and $E=\cup_{i=1}^k E_i'$.
Then we fix a family $\mathscr{G}$ of $q$-color patterns for $G$
such that $g\in\mathscr{G}$
if and only if there exist $i\in [k]$ and $j\in [d_F(v)]$ such that
$g[V_i']\cong f_{1,j}$ 
and such that $g[V_{\ell}']\cong f_2$ for all $\ell \neq i$.

By the definition of the patterns $f_{1,j}$ and $f_2$,
and since the vertex sets $V_i'$ for $i\in [k]$ are pairwise disjoint,
we know that $\mathscr{G}$ is a family of $H$-free $q$-color patterns for $G$. Hence, applying Theorem~\ref{thm:pattern_gadgets_existence}, we can find
a pattern gadget $P=P(H,G,\mathscr{G},q)$
such that $(P,G)$ is $H$-robust. Finally, we obtain $\widetilde{G}$ from $P$ by adding the vertices $v_1,\ldots,v_k$ and by connecting $v_i$ to all vertices in $W_i$ via the edges
$g_1^i,\ldots,g_{d_F(v)}^i$ for all $i\in [k]$.

Analogously to the proof of 
Theorem~\ref{cor:arbitrarily_many_cycles}, 
we now show that $\widetilde{G} \rightarrow_q H$ 
and that each of the edges $g_j^i$, for $i\in [k]$ and $j\in [d_F(v)]$, is important for $\widetilde{G}$ to be Ramsey in the sense that  $\widetilde{G} - g_j^i \not\rightarrow_q H$. 
This then implies the existence of a minimal $q$-Ramsey graph as claimed by the theorem. Indeed, assuming these properties, let $\widetilde{G}' \subseteq \widetilde{G}$ be minimal $q$-Ramsey for $H$. Since $\widetilde{G} - g_j^i \not\rightarrow_q H$,
we can conclude that $g_j^i \in E(\widetilde{G}')$ for every $i \in [k]$ and $j\in [d_F(v)]$. This then implies that 
 $d_{\widetilde{G}'}(v_i)=d_F(v)$.
Hence, $\widetilde{G}'$ is a minimal $q$-Ramsey graph for $H$ with at least $k$ vertices of degree $d_F(v)$.

Let us show first that $\widetilde{G} \rightarrow_{q} H$. 
For a contradiction, suppose we can find
an $H$-free $q$-coloring
$c: E(\widetilde{G})\rightarrow [q]$. For each $i\in[q]$, define $c_i=c^{-1}(i)$ to be the $i$th color class with respect to $c$. By property~\ref{def:patterngadget_somepattern} of the pattern gadget $P$, we know that $g:=\{c_{|G}^{-1}(1), \dots, c_{|G}^{-1}(q)\} \in \mathscr{G}$. Hence, 
by the definition of $\mathscr{G}$, there exist $i \in [k]$ and $j\in [d_F(v)]$ such that $ g[V_i'] \cong f_{1,j}$. 
But then, by the choice of $f_{1,j}$ and the properties of $c_{1,j}$,
we deduce that $c_{|\widetilde{G}[V_i']}$ cannot be extended to an $H$-free $q$-coloring of $\widetilde{G}[V_i'\cup \{v_i\}]$.
This a contradiction, since $c_{|\widetilde{G}[V_i'\cup \{v_i\}]}$
is already such an $H$-free extension by the assumption on $c$.

Next, we show that $\widetilde{G} - g_j^i \not\rightarrow_q H$
for every $i\in [k]$ and $j\in [d_F(v)]$. By symmetry, we may 
only consider the case when $i=j=1$. 
We first partition $G$ in the following way: For every $\ell \in [k]$, we fix a partition $G[V_\ell]= G_{\ell,1} \cup \dots \cup G_{\ell,q}$ such that $\{G_{1,r}\}_{r \leq q} \cong f_{1,1}$ and $\{G_{\ell,r}\}_{r \leq q} \cong f_2$ for $\ell \neq 1$. By the choice of $f_{1,1}$ and $f_2$, we know that the coloring $c: E(G)\rightarrow [q]$ defined by $c(G_{\ell,r})=r$, for every $\ell \in [k]$ and $r \in [q]$, is $H$-free. Moreover, $\{c^{-1}(1), \dots , c^{-1}(q)\} \in \mathscr{G}$ and therefore,
by property~\ref{def:patterngadget_fixedpattern}, we can extend $c$ to an $H$-free $q$-coloring $\varphi_P$ of $P$. 
By the definition of $f_{1,1}$ and the properties of $c_{1,1}$,
we know that the coloring $c|_{G[V_1]}$ can be extended to an $H$-free $q$-coloring $\varphi_1$ of $\widetilde{G}[V_1\cup \{v_1\}]-g_1^1$.
By the definition of $f_{2}$ and the properties of $c_{2}$
we know that, for each $\ell\neq 1$, the coloring $c|_{G[V_\ell]}$ can be extended to an $H$-free $q$-coloring $\varphi_{\ell}$ of 
$\widetilde{G}[V_\ell\cup \{v_\ell\}]$. We now 
put all these colorings together to form
the coloring $\varphi: E(\widetilde{G}-g_1^1)\rightarrow [q]$
given by
$$
\varphi(f) :=
 \begin{cases}
 \varphi_P(f) ~ & \text{ if } f\in E(P) ,\\
 \varphi_\ell(f) ~ & \text{ if } v_\ell\in f
 ~ \text{for some }\ell\in [k] .
 \end{cases}
$$

\medskip

We claim that this coloring is $H$-free.

Assume for a contradiction that there is a monochromatic copy $H'$ of $H$ in the coloring $\varphi$. Then, since $(P,G)$ is $H$-robust,
we know that $H'\subseteq P$ or $H'\subseteq \widetilde{G}[V\cup \{v_\ell\}_{\ell\in [k]}] - g_1^1$. Since the coloring $\varphi_P$ on $P$
is $H$-free, we can assume that $H'\subseteq \widetilde{G}[V\cup \{v_\ell\}_{\ell\in [k]}] - g_1^1$. But then, since $H'$ is connected,
we have $H'\subseteq \widetilde{G}[V_\ell\cup \{v_\ell\}]$ for some $\ell\neq 1$
or $H'\subseteq \widetilde{G}[V_1\cup \{v_1\}] - g_1^1$. In both cases
we know that $H'$ cannot be monochromatic, since
the colorings $\varphi_{1}, \dots, \varphi_k$ are $H$-free. This is a contradiction.
\end{proof}

Finally, we illustrate how to apply Theorem~\ref{thm:arbitrarily_many_3-connected} by deriving Theorem~\ref{cor:arbitrarily_many_cliques} as a consequence of it.

\begin{proof}[Proof of Theorem~\ref{cor:arbitrarily_many_cliques}]
In order to show that $K_t$ is $s_q$-abundant,
it will be enough to prove the existence of a graph $F$
with a vertex $v\in V(F)$ and an edge $e\in E(F)$
satisfying~\ref{goodF:Ramsey}--\ref{goodF:notRamsey2}  with $d_F(v)=s_q(K_t)$.
Implicitly, such a graph is given in
an argument of Fox et al.~\cite{fox2016minimum} 
which was a first step for finding an upper bound on $s_q(K_t)$.
In the following, we will briefly sketch their argument
and then conclude the existence of a graph $F$ as desired.

\medskip

Let $P_q(t-1)$ be the smallest integer $n$ such that 
the following holds:
There exist a graph $G$ on $n$ vertices and a $K_t$-free $q$-color pattern $\set{G_1,\ldots,G_q}$ for $G$ such that,
for every partition $V(G)=\cup_{j\in [q]} V_j$,
there exists a copy $H$ of $K_{t-1}$ and an integer $i\in [q]$ such that $H \subseteq G_i[V_i]$.
Fox et al.\ proved that $s_q(K_t)=P_q(t-1)$ (Theorem~1.5 
in~\cite{fox2016minimum}). For a proof
of the inequality $s_q(K_t) \leq P_q(t-1)$ (Theorem~2.3 
in~\cite{fox2016minimum}), they gave the following construction of a graph $\widetilde{G}$.
\smallskip

Fix a graph $G$ on $P_q(t-1)$ vertices with a $K_t$-free  
$q$-color pattern $\set{G_1,\ldots,G_q}$ as described above.
We take the given graph $G$, an isolated vertex $v$, and a matching $M=\{e_1,\ldots,e_q\}$ that is vertex-disjoint from $G$ and $v$; next, we take a negative signal sender
$S^-:=S^-(K_t,e,f,q,d)$ and a positive signal sender $S^+:=S^+(K_t,e,f,q,d)$ with $d>t$, the existence of which is guaranteed by Theorem~\ref{thm:signal_senders_existence}.
We  then obtain $\widetilde{G}$ as follows:
\begin{enumerate}
\item[(i)] For every distinct $i,j\in [q]$, join $e_i$ and $e_j$ by a copy of $S^-$.
\item[(ii)] For every $i\in [q]$ and every $f\in E(G_i)$, join $e_i$ and $f$ by a copy of $S^+$.
\item[(iii)] Connect $v$ to all vertices in $V(G)$ by an edge.
\end{enumerate}

We will see in the following that $\widetilde{G}\rightarrow_q K_t$, 
$\widetilde{G}-v\not\rightarrow_q K_t$, and 
$\widetilde{G}-M \not\rightarrow_q K_t$. From this, we can then conclude the existence of a graph $F\in \mathcal{M}_q(K_t)$ satisfying the hypothesis of Theorem~\ref{thm:arbitrarily_many_3-connected}.
Indeed, consider any minimal $q$-Ramsey graph $F$ for $K_t$ contained in $\widetilde{G}$.
Since $\widetilde{G}-v \not\rightarrow_q K_t$,
we conclude that $F$ must contain the vertex $v$; moreover, we have $s_q(K_t)\leq d_F(v)\leq d_{\widetilde{G}}(v) = P_q(t-1) = s_q(K_t)$, so $d_F(v) = s_q(K_t)$.
Further, using that 
$\widetilde{G}-M \not\rightarrow_q K_t$,
we also deduce that $\widetilde{G}$ must contain 
at least one edge $e\in M$.
Since $\dist_{F}(v,e)\geq \dist_{\widetilde{G}}(v,e)\geq d>v(K_t)$, $v$ and $e$ cannot share a copy of $K_t$, implying that
property~\ref{goodF:shareH} holds. By the minimality
of $F$, properties~\ref{goodF:Ramsey},~\ref{goodF:notRamsey1},
and~\ref{goodF:notRamsey2} are immediate. We split the remainder of the proof into three claims.

\medskip

\begin{clm}\label{claim:GRamsey}
We have $\widetilde{G}\rightarrow_q K_t$ and
$\widetilde{G}-v\not\rightarrow_q K_t$.
\end{clm}
\textit{Proof.} 
Both statements were already proven in~\cite{fox2016minimum}. We include the argument here for completeness.

We begin by showing that $\widetilde{G} \rightarrow_q K_t$. 
For a contradiction, assume that there exists a $K_t$-free coloring $c:E(\widetilde{G})\rightarrow [q]$. The signal senders in (i) then ensure that the edges of $M$ must receive distinct colors, say without loss of generality that $c(e_i)=i$ for every $i\in [q]$.
The signal senders in (ii) ensure that $c(G_i)=c(e_i)=i$ for every $i\in [q]$. Now, consider the partition $V(G)=\cup_{j\in [q]} V_j$, where, for every $j\in [q]$, we have $w\in V_j$ if and only if $c(vw)=j$. 
Then, by the choice of $G$ and the definition of $P_q(t-1)$,
there exists a graph $H\cong K_{t-1}$
and an integer $i\in [q]$ such that
$H\subseteq G_i[V_i]$.
Hence, the edges in $E(H)\cup \{vw:~ w\in V(H)\}$
all have color $i$ and thus induce a monochromatic copy of $K_t$. This is a contradiction. 

Next, let us show that $\widetilde{G}-v \not\rightarrow_q K_t$. In order to do so, we define a $q$-coloring $c$ of
$\widetilde{G}-v$. We first set
$c(G_i)=c(e_i)=i$ for every $i\in [q]$;
afterwards we extend the coloring $c$ to $\widetilde{G}-v$ in such a way that $c$ is $K_t$-free on each signal sender from (i) and (ii). Note that the latter is possible by property
\ref{def:signal_senders:Hfree} and
\ref{def:signal_senders:signal}.
Analogously to previous proofs, each copy of $K_t$
is fully contained  either in a signal sender or in the graph $G$. Since the coloring restricted to any signal sender is $K_t$-free and since $\set{G_1,\ldots,G_q}$ is a $K_t$-free $q$-color pattern, it follows that $c$ is $K_t$-free.
\hfill \checkmark

\smallskip
The next two claims were not shown in~\cite{fox2016minimum}. 
\begin{clm}\label{claim:Mimportant}
If $s_q(K_t) \leq r_q(K_t)-2$, then $\widetilde{G}-M \not\rightarrow_q K_t$.
\end{clm}
\textit{Proof.} 
In order to see this claim, we define a $q$-coloring $c$ of 
$\widetilde{G}- M$ as follows:
We first fix a $K_t$-free $q$-coloring of $\widetilde{G}[N_{\widetilde{G}}(v)\cup \{v\}] = \widetilde{G}[V(G)\cup \{v\}]$,
which is possible since by assumption
we have
$$
|N_{\widetilde{G}}(v)\cup \{v\}|=s_q(K_t)+1 \leq r_q(K_t)-1~ .
$$
Afterwards, we extend the coloring
to every signal sender so that it is $K_t$-free. The latter is possible since every signal sender is missing at least one signal edge in the graph $\widetilde{G}-M$ 
(and hence we can always pretend that the missing signal edge has a color that fits property~\ref{def:signal_senders:signal}).
Now, each copy of $K_t$
is fully contained either in a signal sender or in the graph $G$,
and hence, the resulting coloring of 
$\widetilde{G}- M$ is $K_t$-free. \hfill \checkmark

\begin{clm}\label{claim:Pqbound}
For all $q\geq 2$ and $t\geq 3$, we have $s_q(K_t) \leq r_q(K_t)-2$.
\end{clm}
\textit{Proof.} Let $t\geq 3$ be fixed. For all $q\geq 2$, define $N_q = (t-1)^q$. To show the claim it suffices to prove that $K_{N_q}$ satisfies the following properties:
\begin{enumerate}
    \item [(i)] There is a $K_t$-free $q$-coloring $\varphi_q$ of $K_{N_q}$ that cannot be extended to a $K_t$-free coloring of $K_{N_q+1}$.
    \item [(ii)] There exists a $K_t$-free coloring $\psi_q$ of $K_{N_q+1}$.
\end{enumerate}
Note that, by an argument similar to that given in Claim~\ref{claim:GRamsey}, property (i) implies that $P_q(t-1) \leq N_q$. Property (ii) implies that $N_q+1 < r_q(K_t)$.  These two inequalities together with the fact that $s_q(K_t) = P_q(t-1)$ imply the claim.
\smallskip

We now proceed by induction on $q$ and show properties (i) and (ii). First consider the case $q=2$. We can use the idea of Burr et al.~\cite{burr1976graphs}. Partition the vertices of the graph $K_{(t-1)^2}$ into $t-1$ equally-sized sets $Q_1,\dots, Q_{t-1}$. Consider the coloring $\varphi_2$ of $K_{(t-1)^2}$ in which the edges lying within a single $Q_i$ are colored red and the edges with endpoints in two different $Q_i$ are colored blue. It is not difficult to check that this coloring is $K_t$-free but there is no way to extend it to $K_{(t-1)^2+1}$ without creating a monochromatic $K_t$, establishing property (i). On the other hand, 
we can define a $K_t$-free 2-coloring $\psi_2$ of $K_{(t-1)^2+1}$ as follows. Let $Q_1,\dots, Q_{t-1}$ be as before; fix an arbitrary vertex $v_i\in Q_i$
for every $i\in [t-1]$. Color all edges of $K_{(t-1)^2}$ as before except for the edge $v_1v_2$, which we now color red. Let $v$ be a new vertex connected to all vertices of $K_{(t-1)^2}$. Color $vv_i$ blue for all $i\in[t-1]$, and color all other edges incident to $v$ red. It is not difficult to check that this coloring is $K_t$-free.

\medskip
Assume that (i) and (ii) hold for some $q\geq 2$.
Consider the graph $K_{N_{q+1}}$. Partition its vertex set into $t-1$ equally-sized sets $Q_1,\dots, Q_{t-1}$. Let $\varphi_{q+1}$ be the coloring in which the edges inside each $Q_i$ are colored according to $\varphi_q$ and the edges between two different $Q_i$ are given color $q+1$. Again, it is easily seen that this coloring is $K_t$-free. Now, let $v$ be a vertex connected to all vertices of $K_{N_{q+1}}$, and consider any coloring of $K_{N_{q+1}+1}$ extending $\varphi_{q+1}$. If all edges from $v$
to some $Q_i$ have colors in $[q]$, then by induction the graph induced by $Q_i\cup\set{v}$ contains a monochromatic copy of $K_t$. So we may assume that, for all $i\in[t-1]$, there is a vertex $v_i\in Q_i$ such that the edge $vv_i$ has color $q+1$. But then the vertices $v_1,\dots, v_{t-1}, v$ induce a monochromatic copy of $K_t$. Hence property (i) is satisfied.
For property (ii), notice that, if $Q_1,\dots, Q_{t-1}$ are as above and $v$ is a new vertex connected to all vertices of $K_{N_{q+1}}$, then coloring the graph induced by $Q_i\cup \set{v}$ according to $\psi_q$ for all $i\in [t-1]$ and giving all edges with endpoints in different $Q_i$ color $q+1$ gives the required $K_t$-free coloring $\psi_{q+1}$ of $K_{N_{q+1}+1}$. \hfill\checkmark
\smallskip

Putting Claims~\ref{claim:GRamsey}--\ref{claim:Pqbound} together, we obtain the theorem.
\end{proof}

\section{Concluding remarks and open problems}\label{sec:conclusion}

In the present paper, we developed a new tool for studying (minimal) Ramsey graphs and showed some applications to questions concerning minimum degrees. In particular, we used pattern gadgets to find examples of graphs $H$ such that a minimal $q$-Ramsey graph for $H$ can contain arbitrarily many vertices of degree $s_q(H)$, that is, $s_q$-abundant graphs. A number of interesting problems remain open.

Questions concerning minimum degrees of minimal Ramsey graphs are particularly interesting for the class of so-called \emph{$q$-Ramsey-simple graphs}. Observe that $s_q(H) \geq q(\delta(H)-1)+1$ for any graph $H$ and integer $q\geq 2$. This was shown by Fox and Lin~\cite{fox2007minimum} for two colors and generalizes easily to any number of colors. Indeed, assume there exists $G\in \mathcal{M}_q(H)$ with a vertex $v \in V(G)$ such that $d_G(v)\leq q(\delta(H)-1)$. Since $G$ is minimal $q$-Ramsey for $H$, we can color the graph $G - v$ with $q$ colors without a monochromatic copy of $H$. Then we can extend this coloring to all of $G$ by coloring at most $\delta(H)-1$ of the edges incident to $v$ in any given color. It is not difficult to check that this is an $H$-free coloring of $G$, a contradiction.
Now, a graph $H$ without isolated vertices is said to be \emph{$q$-Ramsey-simple} if $s_q(H) = q(\delta(H)-1)+1$.  In~\cite{szabo2010minimum}, Szab\'o et al.\ found many classes of $2$-Ramsey-simple bipartite graphs; in particular, all trees were shown to be $2$-Ramsey simple. Later Grinshpun~\cite[Theorems 2.1.2 and 2.1.3]{grinshpun2015some} gave further examples of Ramsey-simple graphs, showing in particular that all 3-connected bipartite graphs are $2$-Ramsey-simple.
Despite this progress, the following question, posed by Szab\'o et al., remains open. 
\begin{question}[\cite{szabo2010minimum}, Problem 2]
Is every bipartite graph with no isolated vertices $2$-Ramsey-simple?
\end{question}
In fact, Grinshpun made the following bolder conjecture.
\begin{conjecture}[\cite{grinshpun2015some}, Conjecture 2.8.2]
Every connected triangle-free graph is $2$-Ramsey-simple.
\end{conjecture}
Some evidence in favor of this conjecture was given in~\cite{grinshpun2017minimum}, where the authors showed that the statement is true for regular 3-connected  triangle-free
graphs satisfying one extra technical condition. It is of course natural to ask the same questions for larger values of $q$. 

In view of the results presented in this paper, it is also interesting to investigate which bipartite (or triangle-free) graphs are $s_q$-abundant. 
This question is particularly interesting for trees. As discussed above, Szab\'o et al.~\cite{szabo2010minimum} showed that, for all trees $T$, we have $s_2(T) = 1$. This result might appear surprising at first, since we might not expect a degree-one vertex to be essential for the Ramsey properties of a graph. Having established that a degree-one vertex can indeed play a significant role in a minimal Ramsey graph for a tree $T$, we might wonder whether we can find many such vertices in a minimal Ramsey graph for $T$. 

It is simple to show that the path $P_4$ with three edges is $s_2$-abundant. Indeed, let $k\geq 3$ be an odd integer and $G$ be the graph obtained from the cycle $C_k$ by adding a distinct pendant edge to each vertex of the cycle. Using the fact that in every 2-coloring of $C_k$ there must be two consecutive edges of the same color, it is not difficult to check that $G$ is a minimal $2$-Ramsey graph for $P_4$. Further, $G$ has $k$ vertices of degree one, establishing the claim.

Thus, we have seen that stars are not $s_2$-abundant but $P_4$ is. For all other trees $T$, the question of whether $T$ is $s_2$-abundant (or, more generally, $s_q$-abundant) remains open. This leads us to propose the following  problem. 

\begin{question}
Let $q\geq 2$ be an integer. Is every tree that is not a star $s_q$-abundant?
\end{question}

As explained above, a positive answer to this question would be rather surprising.
\medskip

More generally, we would like to understand better which graphs $H$ are $s_q$-abundant. We saw in Theorem~\ref{cor:arbitrarily_many_cliques} that we can sometimes show $s_q$-abundance without knowing the precise value of $s_q$. Further, we established a sufficient condition for a given 3-connected graph to be $s_q$-abundant in Theorem~\ref{thm:arbitrarily_many_3-connected}. Given the tools developed in this paper, we believe that all 3-connected graphs should be $s_q$-abundant and propose Conjecture~\ref{conj:3-connected} below.

\begin{conjecture}\label{conj:3-connected}
Every 3-connected graph $H$ is $s_q$-abundant for any integer $q\geq 2$. 
\end{conjecture}

\bibliographystyle{amsplain}
\bibliography{biblio}

\appendix
\section{Existence of indicators}\label{sec:appendix}

In this appendix, we prove the existence of indicators, as claimed in Theorem~\ref{thm:indicators_existence}. The proof is along the same lines as the proofs given in~\cite{burr1977ramseyminimal} and~\cite{clemens2018minimal}. 
Our arguments differ from those in~\cite{burr1977ramseyminimal} and~\cite{clemens2018minimal} in three ways. First, we extend the constructions to cover cases~\ref{thm:indicators_existence:cycles} and~\ref{thm:indicators_existence:pendant}; second, we discuss robustness properties; 
third, we strengthen one of the properties shown in in~\cite{burr1977ramseyminimal} and~\cite{clemens2018minimal} (the one corresponding to property~\ref{def:indicator:non-constant}), which is needed for our application.

\begin{proof}[{Proof of Theorem~\ref{thm:indicators_existence}}]
Without loss of generality, we assume that $d > v(H)$.

We proceed by induction on the number of edges in $F$. The basic construction is the same in all three cases
\ref{thm:indicators_existence:3connected}--\ref{thm:indicators_existence:pendant}; we will see that the special properties we require in the latter two cases follow almost immediately from the properties of the respective signal senders given in Theorem~\ref{thm:signal_senders_existence}.

We will in fact show something stronger: Our indicators will satisfy an additional property, which, following Clemens et al.~\cite{clemens2018minimal}, we call {property $\mathcal{T}$}.
We say that an indicator $I = I^+(H,F,e,d,q)$ satisfies {\em property $\mathcal{T}$} if there is a collection of subgraphs $\set{T_f\subseteq I : f\in E(F)}$ satisfying the following properties:
\begin{enumerate}[label=\itmarab{T}]
    \item \label{def:property_T:intersections} $V(T_f)\cap V(F) = f$ for all $f\in E(F)$.
    \item \label{def:property_T:unions} $V(I) = \bigcup_{f\in E(F)} V(T_f)$ and $E(I) = \bigcup_{f\in E(F)} E(T_f)$.
    \item \label{def:property_T:distances} for all distinct $f_1,f_2\in E(F)$ and all $v\in V(T_{f_1})\cap V(T_{f_2})$, we have either $v\in V(F)$ or $\dist_I(v, F)\geq d$.
\end{enumerate}
Property $\mathcal{T}$ will be useful for showing the required robustness properties.

\smallskip
We begin with the base case $e(F)=2$.
In this case, we will show that our indicators possess one additional property, as given below.
\begin{align}\setlength{\mathindent}{0pt}\label{eq:propertystar}
    \text{If }F = \set{f_1,f_2}\text{ is a matching, then }\dist_I(f_1, f_2) \geq d.\tag{*}
\end{align}

We have two different constructions, one for $q=2$ and a different one for $q>2$. 
We start with the former, which is a slightly modified version of the construction given in~\cite{burr1977ramseyminimal}.

\medskip
For $q=2$, begin with a copy $H_0$ of $H$ and let $e_1, e_2 \in E(H_0)$ be arbitrary except when $H \cong  K_t\cdot K_2$, in which case $e_1$ should not be the pendant edge. Let $e$ be an edge disjoint from $H_0$ and $F$. Say $E(F) = \set{f_1,f_2}$. Let $S^-$ and $S^+$ be a negative and a positive signal sender for $H$ in which the distance between each pair of signal edges is at least $d$ and which satisfy the properties guaranteed by Theorem~\ref{thm:signal_senders_existence}.
%that satisfy the  and in which 
Let $I$ be the graph constructed in the following way:
\begin{enumerate}
    \item [(i)] Connect $f_1$ to every edge in $E(H)\setminus \set{e_1,e_2}$ by a copy of $S^-$.
    \item [(ii)] Join $f_2$ and $e_2$ by a copy of $S^-$.
    \item [(iii)] Join $e_1$ and $e$ by a copy of $S^+$.
\end{enumerate}
We claim that the graph $I$ constructed in this way is a positive indicator with indicator edge $e$ that also satisfies the required additional properties in each of the cases~\ref{thm:indicators_existence:3connected}--\ref{thm:indicators_existence:pendant}.

We first discuss where copies of $H$ in the graph $I$ can be located. Note that  Observations~\ref{obs:indicator_subgraph1} and~\ref{obs:indicator_subgraph2} immediately imply the claimed robustness properties.

\begin{obs}\label{obs:indicator_subgraph1}
Let $H$ be $3$-connected  or a cycle.
Let $I'$ be a graph obtained from $I$ by adding a new vertex set $S$ and any collection of edges within $S\cup V(F)$.
Then every copy of $H$ in $I'$ either lies entirely within one of the signal senders from (ii) or (iii), or is fully contained in $S\cup V(F)$, or is the starting copy $H_0$.
\end{obs}

\textit{Proof.} It is not difficult to see that the claim holds when $H \cong  K_3$, so assume now that $v(H)>3$.
For a contradiction, suppose there is a copy $H'$ of $H$ in $I'$ forming a counterexample. Assume first that $H'$ contains an interior vertex $v$ of one of the signal senders from (ii) or (iii); call this signal sender $S'$. Since the distance between the signal edges of $S'$ is at least $d>v(H')$, we know that $H'$ can only contain vertices from one of the signal edges; call this signal edge $f$. Now, $H'$ is a counterexample, so it needs to contain a vertex $w$ not belonging to $S'$. If $H$ is 3-connected, this is not possible, since removing the edge $f$ disconnects the graph $H'$ (any path from $v$ to $w$ in $I'$ must contain a vertex of one of the signal edges of $S'$). If $H$ is a cycle, then $H'$ needs to contain both vertices of $f$, for otherwise we can disconnect $H'$ by removing a vertex of $f$, contradicting the fact that $H'$ is $2$-connected. But then the vertices in $V(f)\cup \set{v}$ participate in a cycle of length strictly smaller than $v(H)$ in $S'$, contradicting our assumption on the girth of the signal senders.

Hence, we may assume that $H'$ is disjoint from the interior of any of the signal senders. So $H'$ is a subgraph of the graph induced by $S\cup V(F)\cup V(H_0)\cup V(e)$, in which the sets $S\cup V(F)$, $V(H_0)$, and $V(e)$ are all disconnected from each other. Hence no copy of $H$ can use vertices from more than one of these sets, implying the claim. \hfill { $\checkmark$}

\smallskip
The proof of the girth property required in part~\ref{thm:indicators_existence:cycles} is very similar to the proof of Observation~\ref{obs:indicator_subgraph1}.

\smallskip
Using a similar argument, we can show the analogous statement for $H\cong K_t\cdot K_2$, given in the observation below. 
\begin{obs}\label{obs:indicator_subgraph2}
Let $H \cong  K_t\cdot K_2$.
Let $I'$ be a graph obtained from $I$ by adding a new vertex set $S$ and any collection of edges within $S\cup V(F)$.
Then every copy of $K_t$ in $I'$ either lies entirely within one of the signal senders from (ii) or (iii), or is fully contained in $S\cup V(F)$, or is in the starting copy $H_0$.
\end{obs}

\medskip

We now turn our attention to property \eqref{eq:propertystar} and property $\mathcal{T}$. If $F$ is a matching, by the choice of the signal senders used in the construction, we indeed have $\dist_I(f_1,f_2) \geq d$. To verify the latter property, notice that the subgraphs $T_{f_2}$, consisting of the signal sender connecting $f_2$ and $e_2$, and $T_{f_1}$, induced by all remaining vertices together with the vertices of $e_2$ and the vertices in $f_1\cap f_2$, satisfy~\ref{def:property_T:intersections}--\ref{def:property_T:distances}. %the required properties.

\medskip
It remains to show that $I$ satisfies properties~\ref{def:indicator:subgraph}--\ref{def:indicator:non-constant} as well as the additional properties required in part~\ref{thm:indicators_existence:pendant}. 

\ref{def:indicator:subgraph} The first part is clear, since in the construction we do not add any further edges between the vertices of $F$. In each case, the second part of the property follows easily from the fact that $\dist_I(F,e)$ must be at least the distance between the signal edges in the signal senders we attach to $f_1$, $f_2$, and $e$.

\ref{def:indicator:Hfree} For this, consider the following coloring:
\begin{itemize}
    \item Give color 1 to the edges of $F$.
    \item Give color 1 to $e_1$ and color 2 to all other edges of $H_0$.
    \item Give color 1 to $e$.
    \item Extend this coloring to each of the signal senders so that no signal sender contains a monochromatic copy of $H$. In case~\ref{thm:indicators_existence:pendant}
    choose these colorings to be $K_t\cdot K_2$-special.
\end{itemize}
Note that the extension in the last step of the coloring is possible
since the colors for the signal edges are chosen so that
they fit property~\ref{def:signal_senders:signal}. Observe also that $F$ is monochromatic.

We claim that this coloring is $H$-free. Indeed, for parts~\ref{thm:indicators_existence:3connected} and~\ref{thm:indicators_existence:cycles}, Observation~\ref{obs:indicator_subgraph1} implies that every copy of $H$ in $I$ either lies entirely within some signal sender or is the starting copy $H_0$; by our choice of the coloring, none of these copies of $H$ are monochromatic. For~\ref{thm:indicators_existence:pendant}, again
neither the starting copy $H_0$ nor any copy of $H$ that is fully contained within a single signal sender is monochromatic.
Any other copy of $H$
must contain a copy of $K_t$ that touches a signal edge and hence cannot be monochromatic by the choice of the $K_t\cdot K_2$-special $2$-coloring from Theorem~\ref{thm:signal_senders_existence}. 

Further, we use the $K_t\cdot K_2$-special $2$-coloring from Theorem~\ref{thm:signal_senders_existence} to color each of the signal senders, so this coloring of $I$ is also a $K_t\cdot K_2$-special $2$-coloring.

\ref{def:indicator:same} If $F$ is monochromatic in, say, color 1, then by property~\ref{def:signal_senders:signal} of the signal senders, in any $H$-free coloring, all edges in $E(H_0)\setminus\set{e_1}$ have color 2 and $e_1$ and $e$ have the same color. For the coloring to be $H$-free, the edge $e_1$, and hence also $e$, must have color 1.

\ref{def:indicator:non-constant} To justify this property, 
 consider the following coloring:
\begin{itemize}
    \item Give color $\varphi_F(f_i)$ to $f_i$ for both $i\in [2]$.
    \item Give color $k$ to $e_1$, color $\varphi_F(f_1)$ to $e_2$, and
    color $\varphi_F(f_2)$ to all other edges of $H_0$.
    \item Give color $k$ to $e$.
    \item Extend this coloring to each of the signal senders so that no signal sender contains a monochromatic copy of $H$. In case~\ref{thm:indicators_existence:pendant},
    choose these colorings to be $K_t\cdot K_2$-special.
\end{itemize}
Again, the extension in the last step of the coloring is possible
since the colors for the signal edges are chosen so that
they fit property~\ref{def:signal_senders:signal}.
The argument needed to check that this coloring is $H$-free is similar to the one used to verify property~\ref{def:indicator:Hfree} above. Also, it is not hard to see that this also gives a $K_t\cdot K_2$-special $2$-coloring in case~\ref{thm:indicators_existence:pendant}.

\smallskip
We now present the construction for $q>2$ and $e(F)=2$, given in \cite{clemens2018minimal}. Say $E(F) = \set{f_1,f_2}$. Let $\set{e_1,\dots,e_{q-1}}$ be a matching, disjoint from $F$. Let $H_1,\dots, H_{q-1}$ be copies of $H$ that are disjoint from $F$ and  $e_1,\dots, e_{q-1}$ and that all intersect in precisely one fixed edge, which we call $e$. Let $S^+$  and $S^-$ be a positive and a negative signal sender for $H$ respectively in which the distance between the signal edges is at least $d$
and which satisfy the additional properties guaranteed by Theorem~\ref{thm:signal_senders_existence}.
Let $I$ be the graph constructed in the following way:

\begin{enumerate}
    \item [(i)] Connect $f_1$ and $e_i$ by a copy of $S^-$ for all $i\in [q-2]$.
    \item [(ii)] Connect $f_2$ and $e_{q-1}$ by a copy of $S^-$.
    \item [(iii)] Join each pair $e_i, e_j$ for $1\leq i<j\leq q-1$ by a copy of $S^-$.
    \item [(iv)] For all $i\in [q-1]$, connect $e_i$ to all edges of $H_i$ except for $e$ by a copy of $S^+$.
\end{enumerate}

The verification that this indeed gives an indicator is similar to the one in the case $q=2$. As before, every copy of $H$ in parts~\ref{thm:indicators_existence:3connected} and \ref{thm:indicators_existence:cycles} either is one of the starting copies $H_1,\dots, H_{q-1}$ or is fully contained within a single signal sender (similarly  for $K_t$ in part~\ref{thm:indicators_existence:pendant}). 
Notice also that the robustness property, the girth property required in~\ref{thm:indicators_existence:cycles}, and properties $\mathcal{T}$ and~\eqref{eq:propertystar} are shown in a similar way as in the case $q=2$.

\smallskip
Finally, we check properties~\ref{def:indicator:subgraph}--\ref{def:indicator:non-constant}.
Property~\ref{def:indicator:subgraph} is straightforward.

\ref{def:indicator:Hfree} To see this property, color $f_1$, $f_2$, and $e$ with color 1 and, for all $i\in [q-1]$, give $e_i$ and all edges in $E(H_i)\setminus\set{e}$ color $i+1$. Then extend this coloring to all of the signal senders so that each signal sender is colored without a monochromatic copy of $H$, which is possible since the colors chosen above fit property~\ref{def:signal_senders:signal}. By the same argument as in the case $q=2$, there in no monochromatic $H$ anywhere in the graph.

\ref{def:indicator:same} For this, suppose $f_1$ and $f_2$ have the same color, say color 1. Then, in any $H$-free coloring, each of the colors in $[q]\setminus \{1\}$ must be used on the matching $e_1, \dots, e_{q-1}$ exactly once because of the signal senders in (i)--(iii), and each $H_i-e$ needs to be monochromatic in the color of $e_i$ because of the signal senders in (iv). Thus, to avoid a monochromatic copy of $H$, the color of $e$ must be 1, i.e., the same as the color of $F$. 

\ref{def:indicator:non-constant} Suppose that $f_1$ and $f_2$ are colored differently, say using colors 1 and 2 respectively, and we are given any color $k$. If $k\neq 1,2$, we can color $e_1,\dots, e_{q-2}$ with colors $2, \dots, k-1,k+1,\dots, q$ respectively and $e_{q-1}$ with color 1; if $k=2$, we can color $e_1,\dots, e_{q-2}$ with colors $3, \dots, q$ and $e_{q-1}$ with color 1; finally, if $k=1$, we color $e_1,\dots, e_{q-2}$ with colors $2, \dots, q-1$ and $e_{q-1}$ with color $q$.
In each case, color $k$ is available for $e$ and we can still extend the coloring to all the signal senders without creating a monochromatic copy of $H$.

\medskip
We now proceed with the induction step. Suppose there exist indicators as required in the statement of the theorem that also satisfy properties $\mathcal{T}$ and \eqref{eq:propertystar} when $e(F) \leq \ell$ for some $\ell\geq 2$. Assume $e(F) = \ell+1$. Let $f$ be any edge of $F$ and $F' = F - f$; further, let $e'$ and $e$ be two edges that are disjoint from $F$ and from each other. By the induction hypothesis, there exists a positive $(H, F', e', q, d)$-indicator $I'$ satisfying all of the required properties. 
There also exists a positive $(H, \set{e',f}, e, q, d)$-indicator $I''$ in which the distance between $e'$ and $f$ is at least $d$. Now, let $I$ be the graph obtained by joining $F'$ and $e'$ by $I'$ and $\set{f,e'}$ and $e$ by $I''$. We claim that $I$ is a positive indicator satisfying all required properties. 

\medskip
First, as before, we discuss where copies of $H$ can be located. For this, consider a graph obtained from $I$ by adding a new set of vertices $S$ and any edges within $S\cup V(F)$. We claim that in parts~\ref{thm:indicators_existence:3connected} and~\ref{thm:indicators_existence:cycles} every copy of $H$ is contained entirely within  $I'$, $I''$, or the graph induced by $S\cup V(F)$, and that in part~\ref{thm:indicators_existence:pendant} every copy of $K_t$ is contained entirely within $I'$, $I''$, or the graph induced by $S\cup V(F)$. Again, this immediately implies the claimed robustness properties. 

Assume first that $H$ is either 3-connected or isomorphic to a cycle.
Let $H'$ be a copy of $H$ in the new graph. Suppose that $H'$ contains an interior vertex of $I''$. By our assumption on the distance between $f$ and $e'$ in $I''$, the graph $H'$ can only contain vertices from one of these two edges. Again, by the condition that $H$ is 3-connected or isomorphic to $K_3$ in~\ref{thm:indicators_existence:3connected} and by our assumption on the girth of each indicator and of $F$ in~\ref{thm:indicators_existence:cycles}, we conclude that  $H'\subseteq I''$. 

Suppose next that $H'$ contains no such vertices but contains an interior vertex $v$ of $I'$. If $H'$ contains no vertices of $F$, then $H'$ cannot contain any vertices from the set $S$ either, and thus $H'$ is fully contained in $I'$. Hence, we may assume that $H'$ contains a vertex of $F$.
Let $\set{T_{f'}: f'\in E(F')}$ be a collection of subgraphs of $I'$ witnessing that $I'$ has property $\mathcal{T}$. Then $v$ is contained in some $T_{g}$ for $g\in E(F')$; further, this $g$ is unique, since otherwise $\dist_{I}(F, v) \geq d > v(H)$ by~\ref{def:property_T:distances}, leading to a contradiction. 

Since $V(T_g) \cap V(F) = g$, the only possible copy of $K_3$ containing $v$ and a vertex of $F$ consists of $v$ and the endpoints of $g$ and is thus fully contained in $I'$. If $(S\cup V(F))\cap V(H') \subseteq V(g)$, then $H'$ is fully contained in $I'$. So $H'$ contains a vertex from $(S\cup V(F))\setminus V(g)$. If $H$ is 3-connected, this is not possible, since removing the vertices of $g$ disconnects $H'$. Finally, suppose $H$ is a cycle. In this case, by the fact that $H'$ is $2$-connected, both vertices of $g$ must be contained in $H'$, but this means that $v$ and the vertices of $g$ are part of a cycle of length strictly smaller than $v(H')$, contradicting our assumption about the girth of $I'$. 

The argument for part~\ref{thm:indicators_existence:pendant} is analogous to that for part~\ref{thm:indicators_existence:3connected}. Also, a similar argument shows that the girth condition required in part~\ref{thm:indicators_existence:cycles} holds.

Now, to verify property $\mathcal{T}$, consider a collection $\set{T_{f'} : f'\in E(F')}$ of subgraphs of $I'$ given by property $\mathcal{T}$. Adding the graph $I''$, we obtain a collection of subgraphs of $I$ satisfying~\ref{def:property_T:intersections}--\ref{def:property_T:distances}.

Finally, we check the indicator properties.
Property~\ref{def:indicator:subgraph} is clear, since we do not add any new edges within $F$ and 
$\dist_I(F,e) \geq \dist_{I''}(\set{f, e'}, e)\geq d$. 

\ref{def:indicator:Hfree}
The coloring $c$ given by $c(F) = c(e') = c(e) = 1$ can be extended to $I'$ and $I''$ so that neither contains a monochromatic copy of $H$. 
Furthermore, in parts~\ref{thm:indicators_existence:3connected} and~\ref{thm:indicators_existence:cycles}, every copy of $H$ lies entirely within one of $I'$ and $I''$, so there is no monochromatic copy of $H$ in $I$. For part~\ref{thm:indicators_existence:pendant}, the same follows from the fact that every copy of $K_t$ is contained within $I'$ or $I''$ and we can choose the extensions to $I'$ and $I''$ in such a way that no vertex of $F, e',$ or $e$ is in a monochromatic copy of $K_t$. The latter is possible as
we can find $K_t\cdot K_2$-special $2$-colorings for $I'$ and $I''$ by induction. Also, using these $K_t\cdot K_2$-special $2$-colorings we obtain a coloring of the whole graph $I$ such that all edges incident to $e$ have a different color from $e$ and no vertex of $F$ or $e$ is part of a monochromatic $K_t$, as required. Thus, we conclude that~\ref{def:indicator:Hfree} holds and we immediately get the existence of a $K_t\cdot K_2$-special $2$-coloring as required in part~\ref{thm:indicators_existence:pendant}. 

\ref{def:indicator:same} For this, note that, if $F$ is monochromatic in an $H$-free coloring, then $I'$ forces $e'$ to have the same color as $F$, and $I''$ in turn makes it necessary for $e$ to also have the same color as $F$. 

\ref{def:indicator:non-constant} For the last property,
let $\varphi_F$ be any non-constant coloring of the edges of $F$ and $k\in [q]$. Let $f', f''\in E(F)$ be two edges that have distinct colors, say $\varphi_F(f') = i$ and $\varphi_F(f'') = j$ for some distinct $i,j\in[q]$. First assume that $f' = f$. In this case whether $F-f$ is monochromatic or not, there exists an extension of $\varphi_F$ to an $H$-free coloring of $I'$ such that the color of $e'$ is $j$, which in turn means that there is an extension of this coloring also to 
an $H$-free coloring of $I''$ such that the color of $e$ is $k$.
Otherwise, if $f' \in F-f$, then $F-f$ is not monochromatic and there is an extension of $\varphi_F$ to an $H$-free coloring of $I'$ in which $e'$ has color $k$. Then, whether $\set{f,e'}$ is monochromatic or not, there exists an extension of $\varphi_F$ also to an $H$-free coloring of $I''$ such that $e$ has color $k$. Again, it is not difficult to check that this is an $H$-free coloring of $I$ that in case~\ref{thm:signal_senders_existence:pendant} can be made $K_t\cdot K_2$-special.
\end{proof}

\end{document}